\documentclass[10pt,oneside]{article}
\usepackage[utf8]{inputenc}
\usepackage[bookmarks,colorlinks]{hyperref}  
\hypersetup{colorlinks,citecolor={blue}}
\usepackage{amsthm,mathtools}
\usepackage{amssymb}
\usepackage{amsmath}
\usepackage[utf8]{inputenc}
\usepackage[margin=3.2cm]{geometry}
\usepackage{ellipsis}
\usepackage{tikz}

\theoremstyle{plain}
\newtheorem{prop}{Proposition}[section]
\newtheorem{theorem}[prop]{Theorem}

\theoremstyle{definition}
\newtheorem{definition}[prop]{Definition}

\theoremstyle{remark}
\newtheorem{remark}[prop]{Remark}
\newtheorem{example}[prop]{Example}

\title{\textbf{Multi-moment maps on nearly K\"ahler six-manifolds}}
\author{Giovanni Russo}
\date{}

\newcommand{\id}{\mathrm{Id}}
\newcommand{\diag}{\mathrm{diag}}
\newcommand{\Span}{\mathrm{Span}}
\DeclareMathOperator{\chair}{\lrcorner \thinspace}
\DeclareMathOperator{\re}{\mathrm{Re}}
\DeclareMathOperator{\im}{\mathrm{Im}}

\DeclareMathOperator{\tr}{\mathrm{Tr}}
\DeclareMathOperator{\Lie}{\mathcal{L}}

\newcommand\conjugate[1]{\overline{#1}\vphantom{#1}}

\numberwithin{equation}{section}

\begin{document}
\maketitle
\begin{abstract}
We study multi-moment maps on nearly K\"ahler six-manifolds with a two-torus symmetry.
Critical points of these maps have non-trivial stabilisers.
The configuration of fixed-points and one-dimensional orbits is worked out for generic six-manifolds equipped with an $\mathrm{SU}(3)$-structure admitting a two-torus symmetry. 
Projecting the subspaces obtained to the orbit space yields a trivalent graph. We illustrate this result concretely on the 
homogeneous nearly K\"ahler examples.
\end{abstract}

\tableofcontents

\section{Introduction}
\label{introduction}

Nearly K\"ahler manifolds belong to the wider framework of
almost Hermitian geometry. A formal definition was given by Gray in the 1970's, and is the following.
\begin{definition}{\cite{Gray1976}}
\label{definition1.1}
Let $(M,g,J)$ be an almost Hermitian manifold with Riemannian metric $g$ and almost complex 
structure $J$ compatible with $g$. Let $\nabla$ be the Levi-Civita connection. Then $M$ is called \emph{nearly K\"ahler}
if $(\nabla_X J)X = 0$ for every vector field $X$ on $M$.
\end{definition}
Lowering the upper index of $J$ yields a two-form $\sigma \coloneqq g(J{}\cdot{},{}\cdot{})$ on $M$, the \emph{fundamental}
two-form. Hereafter we assume $M$ has dimension six. 
In this case the skew-symmetry of $\nabla J$ is equivalent to the existence of a complex $(3,0)$-form $\psi_{\mathbb{C}} \coloneqq \psi_++i\psi_-$ on $M$ such that, up to homothety
\begin{equation}
\label{nk_structure_equations}
d\sigma = 3\psi_+, \qquad d\psi_- = -2\sigma \wedge \sigma.
\end{equation}
The two identities are the nearly K\"ahler  \emph{structure
  equations}. The above result was first illustrated in the 
special case of the flag manifold of $\mathbb{C}^3$ by Bryant \cite{Bryant}, and later on explained in more generality by Carri\'on \cite{Reyes1993}. 
Another reference is \cite{russo_thesis}.

There are only four compact, homogeneous nearly K\"ahler six-dimensional manifolds \cite{Butruille}: the six-sphere 
$\mathbb{S}^6 = \mathrm{G}_2/\mathrm{SU}(3)$, the flag manifold $F_{1,2}(\mathbb{C}^3) = \mathrm{SU}(3)/T^2$, 
the complex projective space $\mathbb{CP}^3 = \mathrm{Sp}(2)/\mathrm{Sp}(1)\mathrm{U}(1)$, 
and the product of three-spheres $\mathbb{S}^3 \times \mathbb{S}^3 = \mathrm{SU}(2)^3/\mathrm{SU}(2)_{\Delta}$. 
Recently, Foscolo and Haskins \cite{Foscolo} proved the existence of one cohomogeneity-one nearly K\"ahler structure on
$\mathbb{S}^6$ and one on $\mathbb{S}^3 \times \mathbb{S}^3$. In this case the Lie group acting is $\mathrm{SU}(2) \times
\mathrm{SU}(2)$. 

All the spaces above have symmetry rank at least two, 
so it is a sensible question to ask whether there is a theory of nearly K\"ahler six-manifolds with a two-torus symmetry.
A contribution in this direction was given in \cite{russo_swann} making use of \emph{multi-moment maps}. 
 Assume a two-torus $T^2$ acts on a nearly K\"ahler six-manifold $(M,\sigma,\psi_{\pm})$
preserving the $\mathrm{SU}(3)$-structure. The action induces vector fields $U$ and $V$ on $M$, 
thus we have a smooth, $T^2$-invariant, real-valued global function given by
\begin{equation}
\label{multi_moment_map_general}
\nu_M \coloneqq \sigma(U,V),
\end{equation}
and its differential can be computed by Cartan's formula obtaining
\begin{align}
\label{differential}
d\nu_M = 3\psi_+(U,V,{}\cdot{}).
\end{align}
Identity \eqref{differential} and invariance imply $\nu_M$ is a multi-moment map for the torus action \cite{MadsenSwann}. 
In \cite{russo_swann} we showed that level sets of $\nu_M$ corresponding to regular values are principal $T^2$-bundles
over three-dimensional smooth quotients. However, critical sets do not fit into this description, so we need to understand
their structure.

The objective of this article is to study critical sets of multi-moment maps on nearly K\"ahler six-manifolds
with a two-torus symmetry. This integrates the description of general properties of multi-moment maps initiated in \cite{russo_swann} 
and provides concrete examples to work with. The $T^2$-invariance of \eqref{differential} implies that if a point is critical then all points in its orbit are critical as well. 
We then talk about \emph{critical orbits} of $\nu_M$.  
We explain the correspondence between critical orbits and points with non-trivial stabiliser, proving a general
result on the configurations of the latter. Part of this information may be encoded in \textit{trivalent graphs} lying in $\nu_M^{-1}(0)/T^2$, as illustrated in Theorem~\ref{trivalentgraphformalization}. 
Such graphs represent the topological structure of the
critical subspaces of $M$ where the multi-moment map vanishes. A substantial part of our work will be to construct these graphs for all homogeneous nearly K\"ahler six-manifolds.
We then give explicit descriptions of the quotients $\nu_M^{-1}(0)/T^2$
for $M$ homogeneous. The corresponding graphs when $M$ is $\mathbb S^6, F_{1,2}(\mathbb C^3)$, and $\mathbb{CP}^3$,
will represent the set of points in $\nu_M^{-1}(0)/T^2$ where the real part of the torus $T^2$ acts non-freely.
It turns out that when $M$ is homogeneous and $M$ is not $\mathbb{S}^3 \times \mathbb{S}^3$ critical orbits
correspond to extrema of $\nu_M$ or lie in the fibre $\nu_M^{-1}(0)$. The case $\mathbb{S}^3\times \mathbb{S}^3$ admits a $T^3$-symmetry,
so one has infinitely many choices for the two-torus acting, and each yields different pictures.

The paper is organised as follows. Since the discussion of the special cases is fairly technical we give the abstract results first. 
One may then study the homogeneous cases bearing in mind the general framework described in Section \ref{general_results}.
In all the remaining sections we quickly recall the homogeneous structure of each space, introduce two-torus actions,
and then construct multi-moment maps. Critical points are found directly by imposing the condition $\psi_+(U,V,{}\cdot{}) = 0$, 
according to~\eqref{differential}. \\

\textbf{Acknowledgements}. This work is supported by the Danish Council for Independent Research | Natural Sciences 
Project DFF - 6108\kern .05em-00358, and by the Danish National Research Foundation grant DNRF95 (Centre for Quantum Geometry of Moduli Spaces). 
The material contained in this paper is part of my PhD thesis \cite{russo_thesis}. I am grateful to Andrew Swann for his help and constant guidance. 
Further, I wish to thank Andrei Moroianu for useful discussions on the topic. Finally, I thank the reviewer for suggesting significant improvements of the exposition.

\section{General results}
\label{general_results}

As already recalled, a nearly K\"ahler six-manifold is an almost Hermitian manifold $(M,g,J)$ with an $\mathrm{SU}(3)$-structure $(\sigma,\psi_{\mathbb{C}}=\psi_++i\psi_-)$ satisfying
the partial differential equations \eqref{nk_structure_equations}. Each tangent space $T_pM$ is then an $\mathrm{SU}(3)$-module isomorphic to $\mathbb{C}^3$ with its standard 
$\mathrm{SU}(3)$-structure, and $\sigma,\psi_{\mathbb{C}}$ at $p$ are invariant forms of this representation.
There is an orthonormal basis $\{e^i,Je^i\}_{i=1,2,3}$ of $T_p^*M$ such that 
\begin{align}
\label{sigma_model} \sigma & = e^1\wedge Je^1 + e^2\wedge Je^2 + e^3\wedge Je^3, \\
\label{psi_+_model} \psi_+ & = e^1 \wedge e^2 \wedge e^3 - Je^1 \wedge Je^2 \wedge e^3 - e^1\wedge Je^2 \wedge Je^3 - Je^1 \wedge e^2 \wedge Je^3, \\
\label{psi_-_model} \psi_- & = e^1 \wedge e^2 \wedge Je^3 - Je^1 \wedge Je^2 \wedge Je^3 + e^1\wedge Je^2 \wedge e^3 + Je^1 \wedge e^2 \wedge e^3.
\end{align}
The shorthand $e^{ijk}$ for $e^i \wedge e^j \wedge e^k$ will sometimes be used, similarly for differential forms of other degrees. A two-torus $T^2$ acting 
effectively on $M$ and preserving the structure $(\sigma,\psi_{\pm})$ induces a pair of commuting vector fields $U,V$ on $M$, and the function $\nu_M = \sigma(U,V)$ 
is a multi-moment map with differential $d\nu_M = 3\psi_+(U,V,{}\cdot{})$.

In what follows we denote $g(U,U),g(U,V),g(V,V)$ by $g_{UU},g_{UV},g_{VV}$ and define $h^2 \coloneqq g_{UU}g_{VV}- g_{UV}^2$.
We use the same notation pointwise. Consider the basis $\{E_i,JE_i\}_{i=1,2,3},$ of $T_pM$ whose dual basis is $\{e^i,Je^i\}$. Up to acting with a special unitary transformation we have 
\begin{equation}
\label{infinitesimal_generators_pointwise}
U_p = g_{UU}^{1/2}E_1, \qquad g_{UU}^{1/2}V_p = g_{UV}E_1+\nu_MJE_1+\bigl(h^2-\nu_M^2\bigr)^{1/2}E_2.
\end{equation}
Thus \eqref{psi_+_model} and \eqref{infinitesimal_generators_pointwise} imply $d\nu_M = 3\psi_+(U,V,{}\cdot{}) = 3\bigl(h^2-\nu_M^2\bigr)^{1/2}e^3$ pointwise. 
\begin{prop}
\label{characterisation_critical_points} 
Let $(M,\sigma,\psi_{\pm})$ be a nearly K\"ahler six-manifold with a two-torus symmetry. Let $U,V$ be the infinitesimal generators of the action and $p$ be a point in $M$. Then
\begin{enumerate}
\item The vectors $U_p$ and $V_p$ are linearly dependent over $\mathbb{R}$ if and only if $p$ is critical and the multi-moment map $\nu_M$ vanishes at $p$. 
\item The vectors $U_p$ and $V_p$ are linearly dependent over $\mathbb{C}$ if and only if $p$ is critical. This is equivalent to saying $\nu_M^2 = h^2$ at~$p$.
\end{enumerate}
\end{prop}
\begin{proof}
Since $d\nu_M = 3\psi_+(U,V,{}\cdot{})$, the expressions in \eqref{infinitesimal_generators_pointwise} imply that a point $p$ is critical if and only $V_p$ is a linear combination of $U_p$ and $JU_p$. Therefore $g_{UU}V_p = g_{UV}U_p+\nu_M JU_p$. If $\nu_M$ vanishes at $p$ then $U_p$ and $V_p$ are linearly dependent over $\mathbb{R}$. The converse implication is obvious.

The second equivalence follows from the expression of $V_p$ in \eqref{infinitesimal_generators_pointwise}.
\end{proof}

\begin{remark}
Note that the value $0$ lies in the interior of $\nu_M(M)$ by \cite[Proposition 3.2]{russo_swann}.
\end{remark}

Observe that if $U_p$ and $V_p$ are linearly dependent over the reals the stabiliser $H_p$ of $p$ cannot be trivial. Let us describe all possible stabilisers $H_p$. In the following result $M$ need not be nearly K\"ahler.

\begin{theorem}
\label{trivalentgraphformalization}
Let $(M,\sigma,\psi_{\pm})$ be a six-dimensional manifold with an $\mathrm{SU}(3)$-structure admitting a two-torus symmetry. Assume the $T^2$-action is effective on $M$. Let $p$ be a point in $M$ and $H_p$ its stabiliser in $T^2$. 
\begin{enumerate}
\item If $\dim H_p = 2$ then $H_p = T^2$ and there is a neighbourhood $W$ of $p$ in $M$ with the following properties: the stabiliser of each point of $W$ is either trivial or a circle $\mathbb{S}^1 < T^2$, and the set of points in $W$ with one-dimensional stabilisers is a disjoint union of three totally geodesic two-dimensional submanifolds which are complex with respect to $J$ and whose closures only meet at \nolinebreak $p$.
\item If $\dim H_p =1$ then $H_p = \mathbb{S}^1 < T^2$ and there is a neighbourhood $W$ of the orbit through $p$ in $M$ with the following properties: the stabiliser of each point of $W$ is either trivial or $H_p$ and the set of points $\{q \in W : \textrm{Stab}_{T^2}(q) = H_p\}$ is a smooth totally geodesic submanifold of dimension two which is complex with respect to $J$.
\item If $\dim H_p = 0$ and $H_p$ is non-trivial, then $H_p \cong \mathbb{Z}_k$ for some $k > 1$. The $T^2$-orbit $E$ through $p$ is a totally geodesic two-dimensional submanifold, complex with respect to $J$, and there is a neighbourhood $W$ of this orbit where $T^2$ acts freely on $W \setminus E$.
\end{enumerate}
\end{theorem}
\begin{proof}
Let $g \in T^2$ and denote by $\vartheta_g$ the diffeomorphism of $M$ mapping $q$ to $gq$. Its differential $T_p\vartheta_g$ is an isomorphism between $T_pM$ and $T_{gp}M$. In particular, when $g \in H_p$, then $T_p\vartheta_g$ is an automorphism of $T_pM$, and $T_p\vartheta_g \in \mathrm{SU}(3)$ by assumption. Up to conjugation, $T_p\vartheta_g$ is an element of a maximal torus in $\mathrm{SU}(3)$, so for concreteness we take $T_p\vartheta_g = \diag(e^{i\vartheta},e^{i\varphi},e^{-i(\vartheta+\varphi)})$ with respect to the standard basis of $\mathbb{C}^3$.

When $\dim H_p=2$ then $H_p$ is exactly $T^2$ by the Closed Subgroup Theorem, and by the Equivariant Tubular Neighbourhood Theorem there is an open neighbourhood of $p$ equivariantly diffeomorphic to the twisted product
\begin{equation*}
T^2 \times_{T^2} (T_pM/T_p(T^2\cdot p)) \cong T_pM \cong \mathbb{C}^3.
\end{equation*} 
A point $q \neq p$ in this neighbourhood coincides with a vector $X$ in $\mathbb{C}^3$, and by equivariance the requirement $gq = q$ in $M$ translates to $T_q\vartheta_g X = X$ in $\mathbb{C}^3$. Denote $X$ by $(z^1,z^2,z^3) \in \mathbb{C}^3$. Then we look for points fixed by a non-trivial element of the torus by imposing the condition $\diag(e^{i\vartheta}, e^{i\zeta}, e^{-i(\vartheta+\zeta)})\cdot (z^1,z^2,z^3) = (z^1,z^2,z^3)$, where $e^{i\vartheta}, e^{i\zeta} \neq 1$.
One can solve the equation explicitly and find there are three $\mathbb{S}^1$-invariant directions $F_1,F_2,F_3$ corresponding to the standard basis of $\mathbb{C}^3$. Thus the lines $zF_1,zF_2,zF_3$ correspond to three two-dimensional invariant subspaces in $\mathbb{C}^3$ whose points have one-dimensional stabiliser. This proves that points $p$ with stabiliser $T^2$ are isolated when they exist, and there are three two-dimensional, disjoint submanifolds in a neighbourhood of the orbit through $p$ in $M$, intersecting at $p$, and whose points are fixed by one-dimensional stabilisers. The fact that they are totally geodesic follows e.g.\ from \cite[Chapter II, Theorem 5.1]{kobayashi}. 

Assume now $p$ has one-dimensional stabiliser $H_p$. Choosing $U$ in the Lie algebra of $H_p$ and $V$ such that $\Span \{U,V\} = \mathfrak{t}^2$, we have $U_p=0, V_p \neq 0$. So $T_pM \cong \Span\{V_p,JV_p\} \oplus \mathbb{R}^4 \cong \mathbb{C} \oplus \mathbb{C}^2$ orthogonally, and $\mathbb{C}^2$ gets an induced $\mathrm{SU}(2)$-structure. Since $V_p$ and $JV_p$ are $H_p$-invariant, $T_p\vartheta_g \in \mathrm{SU}(2)$ for $g \in H_p$. We claim $H_p \cong \mathbb{S}^1$: the connected component of the identity in $H_p$ is conjugate to $\mathbb{S}^1$, so up to a change of basis its elements are diagonal matrices of the form $\diag (e^{i\alpha},e^{-i\alpha})$. But $T_p\vartheta_g$ and $\diag (e^{i\alpha},e^{-i\alpha})$ for all $\alpha$ commute because $H_p$ is Abelian. Thus $T_p\vartheta_g$ must be diagonal, hence in $\mathbb{S}^1$, and the claim is proved.
Therefore, the orbit through $p$ has a neighbourhood diffeomorphic to 
$$T^2 \times_{\mathbb{S}^1} (T_pM/T_p(T^2\cdot p)) \cong \mathbb{S}^1 \times \mathbb{R}^5 \cong \mathbb{S}^1 \times (\mathbb{R} \oplus \mathbb{C}^2).$$
Call $\mathbb{S}^1_-$ the stabiliser $H_p$, so that $T^2 = \mathbb{S}^1_+ \times \mathbb{S}^1_-$. The torus-action on $\mathbb{S}^1 \times (\mathbb{R} \oplus \mathbb{C}^2)$ can be chosen as follows: $\mathbb{S}^1_+$ acts on $\mathbb{S}^1$, and $\mathbb{S}^1_-$ acts on $\mathbb{R} \oplus \mathbb{C}^2$ trivially on $\mathbb{R}$ and as the standard maximal torus in $\mathrm{SU}(2)$ on $\mathbb{C}^2$. But an element in $\mathbb{S}_-^1$ preserves $JV_p$, so a point $q$ in the neighbourhood $\mathbb{S}^1 \times (\mathbb{R} \oplus \mathbb{C}^2)$ is fixed by an element $\ell$ in the two-torus when the corresponding component in $\mathbb{R} \oplus \mathbb{C}^2$ is fixed, namely $T_q\vartheta_{\ell} X = X$ in $\mathbb{R} \oplus \mathbb{C}^2$. Since the action of $H_p$ on $\mathbb{R}$ is trivial, this condition translates to a condition on $\mathbb{C}^2 \subset \mathbb{C}^3 \cong T_pM$. But $\{0\} < \mathbb{C}^2$ is the only invariant subspace. Thus the set of points with non-trivial stabiliser is an invariant two-dimensional, totally geodesic submanifold containing $p$. 

Finally, when $p$ has zero-dimensional stabiliser $H_p$ there are two invariant independent directions $U_p,V_p \neq 0$. Two cases may occur: either $V_p \in \Span\{U_p,JU_p\}$ or $V_p \not \in \Span\{U_p,JU_p\}$. 

In the former case, $T_pM = \Span\{ U_p,JU_p\} \oplus \mathbb{C}^2$, so $H_p \leq \mathrm{SU}(2)$ is a discrete subgroup. But $H_p$ is compact and Abelian, so it is finite in $\mathrm{SU}(2)$ and is then conjugate to $\mathbb{Z}_k$ for some integer $k \geq 1$. Hence the orbit through $p$ has a neighbourhood diffeomorphic to 
\begin{align*}
T^2 \times_{\mathbb{Z}_k} \mathbb{C}^2 & = (T^2 \times_{\mathbb{Z}_k} \{0\}) \cup (T^2 \times_{\mathbb{Z}_k} (\mathbb{C}^2 \setminus \{0\})) \\
& = (T^2/\mathbb{Z}_k) \cup (T^2 \times_{\mathbb{Z}_k} (\mathbb{C}^2 \setminus \{0\})).
\end{align*}
Now, assume a point $q$ in this neighbourhood is fixed by $\mathbb{Z}_k$. Since the action of $\mathbb{Z}_k$ is trivial on $T^2/\mathbb{Z}_k$ and is free on $\mathbb{C}^2 \setminus \{0\}$, $q$ belongs to $T^2/\mathbb{Z}_k \cong T^2$, so it lies in the orbit of $p$. 

In the case $V_p \not \in \Span\{U_p,JU_p\}$ then $H_p$ fixes all of $T_pM$, so it is a subgroup of $\mathrm{SU}(1) = \{1\}$, and is then trivial.
\end{proof}

\begin{remark}
\label{equivalence}
Note that when $H_p$ has positive dimension the generators of the action are linearly dependent over the reals, whereas when $H_p$ is zero-dimensional and non-trivial they are linearly dependent over the complex numbers. 
\end{remark}
\begin{remark}
As remarked, the theorem applies to the general setting of $\mathrm{SU}(3)$-structures. In particular, part one and two apply to non-compact Calabi--Yau three-folds, which can be deduced from a recent result by Madsen and Swann \cite{madsen_swann}.
The idea has also been explored in the physics literature \cite{aganagic}.
\end{remark}
\begin{remark}
\label{graphs_orbifolds}
Consider the projection $\pi \colon M \to M/T^2$. Theorem \ref{trivalentgraphformalization} implies that by mapping fixed-points and two-submanifolds of points with one-dimensional stabiliser to \nolinebreak $M/T^2$ we obtain \emph{trivalent graphs}, namely graphs where three edges depart from each vertex. More precisely, the graphs lie on the three-manifolds $\nu_M^{-1}(0)/T^2$. In the first two cases $W/T^2$ is homeomorphic to $\mathbb{R}^4$. That $\mathbb{C}^3/T^2$ is homeomorphic to $\mathbb{R}^4$ follows from the homeomorphism between $\mathbb{S}^5/T^2$ and $\mathbb{S}^3$ \cite{hughes_swartz} and by taking the cones on the respective spaces. For the second case the homeomorphism is obtained by looking at $\mathbb{C}^2$ as a cone over $\mathbb{S}^3$ and at the sphere $\mathbb{S}^3$ as a principal $\mathbb{S}^1$-bundle over $\mathbb{S}^2$:
\begin{align*}
\mathbb{S}^1 \times (\mathbb{R} \oplus \mathbb{C}^2)/T^2 & \cong (\mathbb{R} \oplus \mathbb{C}^2)/\mathbb{S}_-^1 \cong \mathbb{R} \times (\mathbb{C}^2/\mathbb{S}^1) \\
& \cong\mathbb{R} \times (C(\mathbb{S}^3)/\mathbb{S}^1) \cong \mathbb{R} \times C(\mathbb{S}^3/\mathbb{S}^1) \\
& \cong \mathbb{R} \times C(\mathbb{S}^2) \cong \mathbb{R}^4.
\end{align*}
In the third case the image of the exceptional orbit is an orbifold point in $M/T^2$. The shape of the graphs for the examples constructed by Foscolo and Haskins \cite{Foscolo} are the same as for the homogeneous cases since the tori act in the same way, but the general critical sets may be different.
\end{remark}

\section{The six-sphere}
\label{the_six_sphere}

Let us recall a few basic concepts from $\mathrm{G}_2$ geometry to treat this case, detailed sources for what we need are e.g.\ \cite{Bryant} and \cite{fernandez_gray}. 
Let $V$ be a seven-dimensional vector space over the reals. Let $\{E_1,\dots,E_7\}$ be some basis and denote by $\{e^1,\dots,e^7\}$ its dual, so as to have an identification $V \cong \mathbb{R}^7$. 
Define the three-form
\begin{equation}
\label{three_form_g2}
\varphi \coloneqq e^{123}+e^{145}+e^{167}+e^{246}-e^{257}-e^{347}-e^{356}.
\end{equation}
The general linear group $\mathrm{GL}(7,\mathbb{R})$ acts by left-multiplication on $V$ and therefore induces canonically an action on $\Lambda^3 V^*$.
One may define the Lie group $\mathrm{G}_2$ as the stabiliser of $\varphi$ in $\mathrm{GL}(7,\mathbb{R})$:
\begin{equation*}
\mathrm{G}_2 \coloneqq \{g \in \mathrm{GL}(7,\mathbb{R}): g\varphi = \varphi\}.
\end{equation*}
The three-form $\varphi$ induces an inner product $\langle {}\cdot{},{}\cdot{}\rangle$ and an orientation on $\mathbb{R}^7$. The two objects in turn induce 
a Hodge star operator ${*}$ on $\mathbb{R}^7$, so we have the four-form
\begin{equation*}
{*}\varphi = e^{4567}+e^{2367}+e^{2345}+e^{1357}-e^{1346}-e^{1256}-e^{1247}.
\end{equation*}
Raising an indicex of $\varphi$ gives a $\mathrm{G}_2$-cross product $P\colon V \times V \to V$, defined by
\begin{equation}
\label{g2_cross_product}
\langle P(X,Y),Z\rangle \coloneqq \varphi(X,Y,Z),
\end{equation}
which in particular satisfies $\lVert P(X,Y) \rVert^2 = \lVert X\rVert^2 \lVert Y\rVert^2 - \langle X,Y \rangle^2$ for every pair $X,Y \in V$.

It is well known that $\mathrm{G}_2$ acts transitively on the six-sphere $\mathbb{S}^6 \subset \mathbb{R}^7 \cong V$ and that the isotropy 
group of $(1,0,\dots,0) \in \mathbb{S}^6$ is isomorphic to the special unitary group $\mathrm{SU}(3)$. This implies $\mathbb{S}^6$ is diffeomorphic to $\mathrm{G}_2/\mathrm{SU}(3)$.
Let $i\colon \mathbb{S}^6 \hookrightarrow \mathbb{R}^7$ be the standard immersion and denote by $g$ the pullback of $\langle {}\cdot{},{}\cdot{}\rangle$ by $i$. 
Call $N$ the unit normal to the six-sphere and define $J \colon \mathbb{R}^7 \to \mathbb{R}^7$ as 
$JX \coloneqq P(N,X)$. Let now $p$ be a point in $\mathbb{S}^6$. From \eqref{g2_cross_product} it follows that $J$ maps $T_p\mathbb{S}^6$ to itself. So one can view $J$ as an endomorphism of each tangent space of $\mathbb{S}^6$, 
and we do so without changing our notations. Another easy consequence of \eqref{g2_cross_product} is that $J$ is $g$-orthogonal if and only if $J^2 = -\id$. On the other hand, polarising the identity $\lVert P(X,Y) \rVert^2 = \lVert X\rVert^2 \lVert Y\rVert^2 - \langle X,Y \rangle^2$ yields the general formula $\langle P(X,Y),P(X,Z)\rangle = \lVert X\rVert^2 \langle Y,Z\rangle-\langle X,Y\rangle \langle X,Z \rangle$, so for $Y,Z$ tangent to the sphere 
\begin{equation*}
g(JY,JZ) = \lVert N \rVert^2 g(Y,Z) - g(N,Y)g(N,Z) = g(Y,Z).
\end{equation*}
Hence $J$ is $g$-orthogonal and $J^2 = -\id$ pointwise, i.e.\ $(\mathbb{S}^6,g,J)$ is an almost Hermitian
manifold.

\begin{prop}
The differential forms $\sigma \coloneqq g(J{}\cdot{},{}\cdot{}), \psi_+ \coloneqq i^*\varphi$, and $\psi_- \coloneqq -i^*(N \chair {*}\varphi)$ give a nearly K\"ahler structure on the six-sphere.
\end{prop}
\begin{proof}
The result is not new (cf.\ e.g.\ \cite{Butruille}), we prove it for completeness. One can perform the calculations on $\mathbb{R}^7$ and then restrict the results to the sphere. Let $(x^k)_{k=1,\dots,7}$,
be global coordinates on $\mathbb{R}^7$. Let the one-form $dx^k$ be the dual of the coordinate vector field $\partial/\partial x^k$ for all $k$.
The two-form $\langle J{}\cdot{},{}\cdot{}\rangle $ turns out to have the following shape:
\begin{align}
\label{sigma_six_sphere}
\langle J{}\cdot{},{}\cdot{}\rangle & = x^3 dx^{12}-x^2 dx^{13} +x^5 dx^{14} - x^4 dx^{15}+x^7 dx^{16}-x^6 dx^{17} \nonumber \\
& \qquad +x^1 dx^{23} + x^6 dx^{24} -x^7 dx^{25} - x^4 dx^{26} + x^5 dx^{27} -x^7 dx^{34} \nonumber \\
& \qquad - x^6 dx^{35} + x^5 dx^{36} + x^4 dx^{37} +x^1 dx^{45} + x^2 dx^{46} - x^3 dx^{47} \nonumber \\
& \qquad -x^3 dx^{56} - x^2 dx^{57} + x^1dx^{67}.
\end{align}
A direct computation of its differential gives $d \langle J{}\cdot{},{}\cdot{}\rangle = 3\varphi$.
Pulling back this identity to $\mathbb{S}^6$ yields $d\sigma = 3\psi_+$. Further, $d \psi_- = -i^*d(N \lrcorner \thickspace {*}\varphi)$. 
By the expression of ${*}\varphi$ it follows that $d {}(N \chair {*}\varphi) = 4{*}\varphi$ and again the restrictions to $\mathbb{S}^6$ are equal. Thus the claim is $4i^*{*}\varphi = 2\sigma \wedge \sigma$. Up to a rotation in $\mathrm{G}_2$ mapping $p$ to $E_7$, and so $N$ to $\partial/\partial x^7$, we have 
\begin{equation*}
\sigma = N \chair \varphi = \partial/\partial x^7 \chair \varphi = dx^{16}-dx^{25}-dx^{34},
\end{equation*}
so $\sigma \wedge \sigma = -2(dx^{1256}+dx^{1346}-dx^{2345})$, and $i^*{*}\varphi = dx^{2345}-dx^{1346}-dx^{1256}$ is unchanged. Finally $d \psi_- = -2\sigma \wedge \sigma$, and this proves our claim.
\end{proof}

Now let a two-torus act on $\mathbb{R}^7 \cong \mathbb{C}^3 \oplus \mathbb{R}$ as follows. Take the maximal torus $T^2$ inside $\mathrm{SU}(3)$ given by matrices of the form $A_{\vartheta,\phi} \coloneqq \diag \bigl(e^{i\vartheta},e^{i\phi}, e^{-i(\vartheta+\phi)}\bigr)$, and let $A_{\vartheta,\phi}$ act effectively on the left on $(z^1,z^2,z^3,t) \in \mathbb{C}^3 \oplus \mathbb{R}$ as
\begin{equation*}
A_{\vartheta,\phi}(z^1,z^2,z^3,t) \coloneqq (e^{i\vartheta}z^1,e^{i\phi}z^2,e^{-i(\vartheta+\phi)}z^3,t).
\end{equation*}
Because of the convention chosen for \eqref{three_form_g2} we set $z^1 = x^1+ix^6, z^2 = x^5+ix^2, z^3 = x^4+ix^3$, and $t = x^7$. At $p=(z^1,z^2,z^3,t)$, the fundamental vector fields have the form
\begin{align*}
U_p & = -x^6\partial/\partial x^1-x^4\partial/\partial x^3+x^3\partial/\partial x^4+x^1\partial/\partial x^6, \\
V_p & =x^5\partial/\partial x^2-x^4\partial/\partial x^3+x^3\partial/\partial x^4-x^2\partial/\partial x^5.
\end{align*}
Plugging the two vectors in the two-form \eqref{sigma_six_sphere} and restricting to the six-sphere, one finds the multi-moment map
\begin{equation*}
\nu_{\mathbb{S}^6}(p) = 3\bigl(x^1(x^4x^5-x^2x^3)-x^6(x^3x^5+x^2x^4)\bigr).
\end{equation*} 
Using complex coordinates the $T^2$-invariance is evident:
\begin{equation}
\label{multi_moment_map_six_sphere_expression}
\nu_{\mathbb{S}^6}(p) = 3(\re(z^1)\re(z^2z^3)-\im(z^1)\im(z^2z^3)) = 3\re(z^1z^2z^3).
\end{equation}

\begin{prop}
\label{identification_quotient_sphere}
The orbit space $\nu_{\mathbb S^6}^{-1}(0)/T^2$ can be identified with $\mathbb{S}^3/\mathbb{Z}_2^2$.
\end{prop}
\begin{proof}
Let $p=(z^1,z^2,z^3,t)$ be a point on $\mathbb S^6$ such that $\nu_{\mathbb S^6}(p)=0$. Up to acting on $p$ with the torus-action, we can always assume two of its complex coordinates to be real and non-negative. 
Therefore, the expression \eqref{multi_moment_map_six_sphere_expression} implies that $\nu_{\mathbb S^6}(p)=0$ when the remaining complex coordinate is purely imaginary.
Hence there is always a representative $T^2\cdot p \in \nu_{\mathbb S^6}^{-1}(0)/T^2$ lying on the three-sphere $t^2+\sum_{i=1}^3 \lvert z^i\rvert^2=1$. The real part of $T^2$ acting on it is isomorphic to $\mathbb{Z}_2^2$, and this proves the result.
\end{proof}

\begin{prop}
\label{critical_sets_six_sphere}
There are three two-dimensional spheres in $\mathbb{S}^6$ where $\nu_{\mathbb{S}^6}$ and $d\nu_{\mathbb{S}^6}$ vanish. These intersect at two common points. Further, there are two $T^2$-orbits
of critical points where $\nu_{\mathbb{S}^6}$ attains its extrema.
\end{prop}
\begin{proof}
We have $d \nu_{\mathbb{S}^6} = 3\psi_+(U,V,{}\cdot{}) = 3(i^*\varphi)(U,V,{}\cdot{}) = 3g(P(U,V),{}\cdot{})$, which vanishes if and only if $P(U,V)$ is parallel to $N$. Projecting $P(U,V)$ on each $\partial/\partial x^k$ one finds 
\begin{align*}
P(U,V) & = (x^4x^5-x^2x^3)\partial/\partial x^1 - (x^4x^6+x^1x^3)\partial/\partial x^2 \\
& \qquad - (x^5x^6+x^1x^2)\partial/\partial x^3 + (x^1x^5-x^2x^6)\partial/\partial x^4 \\
& \qquad + (x^1x^4-x^3x^6)\partial/\partial x^5 - (x^2x^4+x^3x^5)\partial/\partial x^6. 
\end{align*}
To see where $P(U,V)$ is proportional to $N = x^1\partial/\partial x^1+\cdots+x^7\partial/\partial x^7$ we need to find points $(x^1,\dots,x^7)$ such that the following equations hold for some real proportionality factor $\lambda$:
\begin{alignat*}{3}
\lambda x^7=0, \qquad & x^4x^5-x^2x^3 = \lambda x^1, \qquad && x^1x^5-x^2x^6 = \lambda x^4 \\
 \qquad & x^4x^6+x^1x^3 = -\lambda x^2, \qquad && x^1x^4-x^3x^6 = \lambda x^5, \\
 \qquad & x^5x^6+x^1x^2 = -\lambda x^3, \qquad && x^2x^4+x^3x^5 = -\lambda x^6.
\end{alignat*}

The case $\lambda = 0$ gives points $p$ where $P(U,V)=0$, i.e.\ $U_p$ and $V_p$ are linearly dependent over \nolinebreak $\mathbb{R}$. By comparing the expressions of $U_p,V_p$ one obtains three two-spheres of critical points where the multi-moment map vanishes:
\begin{enumerate}
\item $(x^1)^2+(x^6)^2+(x^7)^2=1$ and $x^2=x^3=x^4=x^5=0$.
\item $(x^2)^2+(x^5)^2+(x^7)^2=1$ and $x^1=x^3=x^4=x^6=0$.
\item $(x^3)^2+(x^4)^2+(x^7)^2=1$ and $x^1=x^2=x^5=x^6=0$.
\end{enumerate}
Note that the three spheres have the poles $(x^1,\dots,x^6,x^7) = (0,\dots,0,\pm 1)$ in common. 

Now let us switch to the case $\lambda \neq 0$. Assume we are at a critical point, so in particular $x^7=0$. Up to the action of $U$ and \nolinebreak $V$, we can assume $x^1 = x^2 = 0$ and $x^5,x^6 \geq 0$. The equations characterising critical points yield $x^4 = 0$ and $x^5x^6 = -\lambda x^3, x^3x^6 = -\lambda x^5, x^3x^5 = -\lambda x^6$. If one among $x^3,x^5,x^6$ vanishes, so do all the others, and we get a contradiction as we need solutions on the six-sphere. Therefore we can assume without loss of generality all of them non-zero, which gives $(x^i)^2 = \lambda^2$ for $i = 3,5,6$. We thus obtain two stationary $T^2$-orbits where $\nu_{\mathbb{S}^6}$ attains its maximum and minimum, and $\nu_{\mathbb{S}^6}(\mathbb{S}^6) = [-1/ \sqrt{3},1/ \sqrt{3}]$. 
\end{proof}

The three two-spheres found can be recovered by looking for points with non-trivial stabilisers according to Proposition \ref{characterisation_critical_points}. This is done by solving the equation
\begin{equation*}
A_{\vartheta,\phi}(z^1,z^2,z^3,t) = (z^1,z^2,z^3,t),
\end{equation*}
namely the equations $e^{i\vartheta}z^1 = z^1, e^{i\phi}z^2 = z^2, e^{-i(\vartheta + \phi)}z^3 = z^3$ with $(e^{i\vartheta}, e^{i\phi}) \neq (1,1)$. A discussion of the 
different cases yields the solutions
$$
\begin{cases}
(0,0,0,\pm 1), & \text{ poles fixed by all of } T^2, \\
t^2 + \lvert z^i \rvert^2 = 1, i = 1,2,3, & \text{ two-spheres of points fixed by a circle}.
\end{cases}
$$
Note the three two-spheres correspond to those found in the proof of Proposition \ref{critical_sets_six_sphere}. Projecting the latter to the orbit space $\mathbb{S}^6/T^2$ gives a graph of two points and three edges. The graph represents the set of points where $\mathbb Z_2^2$ in Proposition \ref{identification_quotient_sphere} acts non-freely. For $|z^i| \to 0$ the two-spheres meet at the common poles. Moreover, the spheres do not intersect each other at any point but the poles, so the edges of the graph are disjoint (see Figure \ref{fig:S6}). 

\begin{figure}
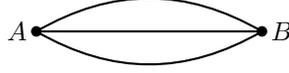

  \centering
  \tikzpicture
    \coordinate (A) at (0,0);
    \coordinate (B) at (3,0);
    \draw[thick][black] (A) -- (B);
    \draw[thick][black] (A) to[bend right] (B);
    \draw[thick][black] (A) to[bend left] (B);
    \fill (A) circle [radius=2pt] node[left] {$A$};
    \fill (B) circle [radius=2pt] node[right] {$B$};
  \endtikzpicture
  \caption{The three two-spheres in $\mathbb{S}^6/T^2$}
  \label{fig:S6}
\end{figure}

\section{The flag manifold}
\label{flag_manifold}

Let $F_{1,2}(\mathbb{C}^3)$ be the set of pairs $(L,U)$ of subspaces in $\mathbb{C}^3$, where $L$ is a complex line contained in the complex plane $U$. Such pairs are called \textit{flags}. The special unitary group $\mathrm{SU}(3)$ acts transitively on $F_{1,2}(\mathbb{C}^3)$. Let $F_1,F_2,F_3$ be the standard basis of $\mathbb{C}^3$. It turns out that the isotropy group of the point $(\langle F_1 \rangle, \langle F_1,F_2 \rangle)$ is isomorphic to the diagonal two-torus \nolinebreak $T^2$, hence $F_{1,2}(\mathbb{C}^3)$ is a smooth manifold diffeomorphic to $\mathrm{SU}(3)/T^2$, and is called \emph{flag manifold} of $\mathbb{C}^3$. 

We now equip $\mathrm{SU}(3)/T^2$ with an almost Hermitian structure. A matrix $p \in \mathrm{SU}(3)$ acts on $\mathrm{SU}(3)$ by left translation and induces a pullback map $(p^{-1}_{\id})^* \colon \mathfrak{su}^*(3)^{\otimes  2} \to T_p^*\mathrm{SU}(3)^{\otimes  2}$. We can thus define a Riemannian metric $g$ on $\mathrm{SU}(3)$ such that
\begin{equation}
\label{metric_flag_pointwise}
g_p \coloneqq \re((p^{-1}_{\id})^*g_0),
\end{equation}
where $g_0$ is the Killing form on $\mathfrak{su}(3)$ normalised as $g_0(X,Y) \coloneqq (1/2)\tr (^t\conjugate{X}Y)$.
The metric $g$ is bi-invariant for $g_0$ is. In particular $g$ is invariant under the action of the maximal torus in $\mathrm{SU}(3)$ above, so descends to a metric on the flag manifold, which we still denote by $g$. To construct an almost complex structure $J$ we follow Gray \nolinebreak \cite[Section 3]{gray_symmetric}. Let $A \coloneqq \diag(e^{2\pi i/3}, e^{4\pi i/3}, 1)$ and define the conjugation map $\tilde{\vartheta} \colon \mathrm{SU}(3) \to \mathrm{SU}(3)$ so that $\tilde{\vartheta}(B) = ABA^{-1}$. It is clear by this definition that $\tilde{\vartheta}^3= \id$ (where the cubic exponential stands for composing three times) and that $\tilde{\vartheta}$ fixes the maximal torus $T^2$ in $\mathrm{SU}(3)$ above. So $\tilde{\vartheta}$ induces a map on the quotient $\vartheta \colon \mathrm{SU}(3)/T^2 \to \mathrm{SU}(3)/T^2$ that fixes the coset $T^2$ and satisfies $\vartheta^3 = \id$. We define $J_0$ at the identity as follows: for $X \in \mathfrak{su}(3)/\mathfrak{t}^2$ write $d\vartheta(X) = AXA^{-1} \eqqcolon -(1/2)X+(\sqrt{3}/2)J_0X$, so that 
\begin{equation}
\label{almost_complex_structure_flag}
J_0X = \tfrac{2}{\sqrt{3}}\bigl(AXA^{-1}+\tfrac{1}{2}X\bigr).
\end{equation}
The map $J_0\colon \mathfrak{su}(3)/\mathfrak{t}^2 \to \mathfrak{su}(3)/\mathfrak{t}^2$ is well defined as $A$ commutes with diagonal matrices. We now check that $J_0^2 = -\id$. Firstly, observe that $d \vartheta - \id$ is injective: if $AXA^{-1}-X = 0$, then $AX = XA$, and since $X$ is diagonalisable then $X$ is diagonal. Thus $X=0$ in $\mathfrak{su}(3)/\mathfrak{t}^2$ and $\id - d \vartheta$ is left-invertible. This amounts to say that $0 = \id + d \vartheta + d \vartheta^2$, or more explicitly that $X+AXA^{-1}+A^2XA^{-2} = 0$. Therefore
\begin{align*}
J_0^2X & = \tfrac{4}{3}(A(AXA^{-1}+\tfrac{1}{2}X )A^{-1} + \tfrac{1}{2}(AXA^{-1}+\tfrac{1}{2}X)) \\
& = \tfrac{4}{3}(A^2XA^{-2}+AXA^{-1}+X-\tfrac{3}{4}X) = -X,
\end{align*}
as we wanted. A similar computation shows $J_0$ is an isometry. We can move the operator $J_0$ to every point $p \in \mathrm{SU}(3)$ so that for each $Y \in T_p(\mathrm{SU}(3)/T^2)$ one has $J_p(Y) = pJ_0(p^{-1}Y)$. From the invariance of $g$ and $J$ it follows that $(g,J)$ is an almost Hermitian structure on the flag manifold of $\mathbb{C}^3$. 

To construct a nearly K\"ahler structure using $g$ and $J$, we work at the identity of $\mathrm{SU}(3)$ and define explicit basic forms $\sigma_0,\varphi_0,\psi_0$ satisfying $d\sigma_0 = 3\varphi_0$ and $d\psi_0 = -2\sigma_0 \wedge \sigma_0$. Finally we extend this structure to the whole flag manifold. A basis of \nolinebreak $\mathfrak{su}(3)$ is given by the matrices
\begin{alignat*}{4}
& E_1 = \left(
\begin{smallmatrix}
0 & i & 0 \\
i & 0 & 0 \\
0 & 0 & 0
\end{smallmatrix}
\right), \quad
&& E_2 = \left(
\begin{smallmatrix}
0 & 1 & 0 \\
-1 & 0 & 0 \\
0 & 0 & 0
\end{smallmatrix}
\right), \quad 
&& E_3 = \left(
\begin{smallmatrix}
0 & 0 & 1 \\
0 & 0 & 0 \\
-1 & 0 & 0
\end{smallmatrix}
\right), \quad 
&& E_4 = \left(
\begin{smallmatrix}
0 & 0 & i \\
0 & 0 & 0 \\
i & 0 & 0
\end{smallmatrix}
\right), \\
& E_5 = \left(
\begin{smallmatrix}
0 & 0 & 0 \\
0 & 0 & i \\
0 & i & 0
\end{smallmatrix}
\right), \quad 
&& E_6 = \left(
\begin{smallmatrix}
0 & 0 & 0 \\
0 & 0 & 1 \\
0 & -1 & 0
\end{smallmatrix}
\right), \quad 
&& E_7 = \left(
\begin{smallmatrix}
i & 0 & 0 \\
0 & 0& 0 \\
0 & 0 & -i
\end{smallmatrix}
\right), \quad 
&& E_8 = \left(
\begin{smallmatrix}
0 & 0 & 0 \\
0 & i & 0 \\
0 & 0 & -i
\end{smallmatrix}
\right).
\end{alignat*}
Denote by $e^k$ the dual of $E_k$. Using \eqref{metric_flag_pointwise} and \eqref{almost_complex_structure_flag} one can check that 
\begin{align}
\label{metric_flag} g_0 & = e^1 \otimes  e^1 + \cdots + e^6 \otimes e^6, \\
\label{almost_complex_structure_flag_origin} J_0 & = E_2 \otimes  e^1-E_1 \otimes  e^2+E_4 \otimes  e^3 -E_3 \otimes  e^4 + E_6 \otimes  e^5-E_5 \otimes  e^6.
\end{align} 
The results for the differentials of $e^1,\dots,e^6$ are
\begin{alignat}{2}
\label{differentials_flag}
de^1 &= e^{46}-e^{35}+e^{27}-e^{28}, \qquad && de^2 = e^{36} + e^{45} - e^{17}+e^{18}, \nonumber \\
de^3 & = e^{15}-e^{26} - 2e^{47}-e^{48}, \qquad  && de^4 = e^{52} + e^{61}+2e^{37}+e^{38}, \nonumber \\
de^5 & = e^{24}-e^{13}+e^{67}+2e^{68}, \qquad  && de^6 = e^{23} + e^{14}-e^{57}-2e^{58}.
\end{alignat}
Moreover $g_0$ and $J_0$ descend to the quotient $\mathfrak{su}(3)/\mathfrak{t}^2$.
\begin{prop}
\label{nk_structure_flag}
The forms on $\mathfrak{su}(3)$ given by
\begin{gather*}
\sigma_0 \coloneqq g_0(J_0{}\cdot{},{}\cdot{}) = e^{12} + e^{34} + e^{56}, \\
\varphi_0 \coloneqq -e^{136} + e^{246} - e^{235} - e^{145}, \qquad \psi_0 \coloneqq e^{135}-e^{245}-e^{146}-e^{236},
\end{gather*}
descend to the quotient $\mathfrak{su}(3)/\mathfrak{t}^2$ and satisfy $d\sigma_0 = 3\varphi_0$, $d\psi_0 = -2\sigma_0 \wedge \sigma_0$. As a consequence, the differential forms on $\mathrm{SU}(3)/T^2$ given by $\sigma_p = g_p(J_p{}\cdot{},{}\cdot{}),
\psi_{+|p} \coloneqq \varphi_0(p^{-1}{}\cdot{},p^{-1}{}\cdot{},p^{-1}{}\cdot{})$, and $\psi_{-|p} \coloneqq \psi_0(p^{-1}{}\cdot{},p^{-1}{}\cdot{},p^{-1}{}\cdot{})$, define a nearly K\"ahler structure on $F_{1,2}(\mathbb{C}^3)$.
\end{prop}
\begin{proof}
First of all, $\sigma_0,\varphi_0,\psi_0$ descend to the quotient because $\Lie_{E_k} \sigma_0 = 0 = \Lie_{E_k} \varphi_0 = \Lie_{E_k} \psi_0$, for $k = 7,8$, and their contractions with $E_7,E_8$ vanish. The results in \eqref{differentials_flag} imply $d\sigma_0 = 3\varphi_0$, $d\psi_0 =-2\sigma_0 \wedge \sigma_0$. The last part of the statement follows by the invariance of the latter equations under translation.
\end{proof}

Consider now the maximal torus $T^2$ in $\mathrm{SU}(3)$ given by the matrices $\diag(e^{i\vartheta},e^{i\varphi},e^{-i(\vartheta+\varphi)})$. Two linearly independent generators of its Lie algebra are $U = \diag(-i,2i,-i), V = \diag(-i,-i,2i)$, 
and at $p \in \mathrm{SU}(3)$ they induce infinitesimal generators of the action $U_p = Up$ and $V_p = Vp$. Thus $p^{-1}Up, p^{-1}Vp$ are vectors in the Lie algebra $\mathfrak{su}(3)$, which splits as $\mathfrak{t}^2 \oplus \mathfrak{m}$, $\mathfrak{m}$ containing matrices with zeros on the diagonal. So when we work on the quotient $\mathrm{SU}(3)/T^2$ we need to take the projections $(p^{-1}Up)_{\mathfrak{m}}$ and $(p^{-1}Vp)_{\mathfrak{m}}$. A matrix $p = (p_{ij})_{i,j=1,2,3} \in \mathrm{SU}(3)$ has determinant \nolinebreak $1$ and satisfies the conditions
\begin{alignat*}{2}
&\lvert p_{11} \rvert^2 + \lvert p_{21}\rvert^2 + \lvert p_{31}\rvert^2 = 1, \qquad && \conjugate{p}_{11}p_{12}+\conjugate{p}_{21}p_{22}+\conjugate{p}_{31}p_{32}=0,\\
&\lvert p_{12} \rvert^2 + \lvert p_{22}\rvert^2 + \lvert p_{32}\rvert^2 = 1, \qquad && \conjugate{p}_{11}p_{13}+\conjugate{p}_{21}p_{23}+\conjugate{p}_{31}p_{33}=0,\\
&\lvert p_{13} \rvert^2 + \lvert p_{23}\rvert^2 + \lvert p_{33}\rvert^2 = 1, \qquad && \conjugate{p}_{12}p_{13}+\conjugate{p}_{22}p_{23}+\conjugate{p}_{32}p_{33}=0.
\end{alignat*}
One can compute explicitly $p^{-1}Up, p^{-1}Vp$ and project them onto $\mathfrak{m}$: 
\begin{alignat*}{2}
(p^{-1}Up)_{\mathfrak{m}} & = \left(
\begin{matrix}
 & iz^1 & iz^2 \\
 i\conjugate{z}^1 & & iz^3 \\
 i\conjugate{z}^2 & i\conjugate{z}^3 & 
\end{matrix}
\right), \quad 
&&
\begin{cases}
z^1 = 3\conjugate{p}_{21}p_{22}, \\
z^2 = 3\conjugate{p}_{21}p_{23}, \\
z^3 = 3\conjugate{p}_{22}p_{23},
\end{cases}
\\
(p^{-1}Vp)_{\mathfrak{m}} & = 
\left(
\begin{matrix}
 & iw^1 & iw^2 \\
 i\conjugate{w}^1 & & iw^3 \\
 i\conjugate{w}^2 & i\conjugate{w}^3 & 
\end{matrix}
\right), \quad
&& 
\begin{cases}
w^1 = 3\conjugate{p}_{31}p_{32}, \\
w^2 = 3\conjugate{p}_{31}p_{33}, \\
w^3 = 3\conjugate{p}_{32}p_{33}.
\end{cases}
\end{alignat*}
Note that the coefficients $z^i$ and $w^k$ are all $T^2$-invariant, where here $T^2$ is the torus acting on the left. We write $(p^{-1}Up)_{\mathfrak{m}}$ and $(p^{-1}Vp)_{\mathfrak{m}}$ in terms of the basis of $\mathfrak{su}(3)/\mathfrak{t}^2$:
\begin{align}
\label{generator_one_action_flag}
(p^{-1}Up)_{\mathfrak{m}} & = \re z^1 E_1 - \im z^1 JE_1 - \im z^2 E_3\nonumber \\
& \qquad + \re z^2 JE_3 + \re z^3 E_5 - \im z^3 JE_5, \\
\label{generator_two_action_flag}
(p^{-1}Vp)_{\mathfrak{m}} & = \re w^1 E_1 - \im w^1 JE_1 - \im w^2 E_3 \nonumber\\
& \qquad + \re w^2 JE_3 + \re w^3 E_5 - \im w^3 JE_5.
\end{align}
The multi-moment map is then $\nu_{F_{1,2}(\mathbb{C}^3)}(p) = \sigma_0((p^{-1}Up)_{\mathfrak{m}}, (p^{-1}Vp)_{\mathfrak{m}})$, namely
\begin{align*}
\nu_{F_{1,2}(\mathbb{C}^3)}(p) & = -\re z^1 \im w^1 + \re w^1 \im z^1 - \im z^2 \re w^2 \nonumber \\
& \qquad + \re z^2 \im w^2 - \re z^3 \im w^3 + \re w^3 \im z^3 \nonumber \\
& = -\im (\conjugate{z}^1w^1-\conjugate{z}^2w^2+\conjugate{z}^3w^3) \\
& = -9\im (p_{21}\conjugate{p}_{22}\conjugate{p}_{31}p_{32}-p_{21}\conjugate{p}_{23}\conjugate{p}_{31}p_{33}+p_{22}\conjugate{p}_{23}\conjugate{p}_{32}p_{33}).
\end{align*}
Since $p^{-1} = {^t}\conjugate{p} \in \mathrm{SU}(3)$ then in particular $p_{21}\conjugate{p}_{31}+p_{22}\conjugate{p}_{32}+p_{23}\conjugate{p}_{33}=0$, so we can simplify our expression to get
\begin{align}
\label{multi_moment_map_flag_expression}
\nu_{F_{1,2}(\mathbb{C}^3)}(p) & = -9\im (-\conjugate{p}_{22}p_{32}p_{23}\conjugate{p}_{33}+2p_{22}\conjugate{p}_{32}\conjugate{p}_{23}p_{33}) \nonumber \\
& = -27\im(p_{22}\conjugate{p}_{23}\conjugate{p}_{32}p_{33}) \nonumber \\
& = 3\im(z^3\conjugate{w}^3).
\end{align}
Since $z^i,w^k$ are $T^2$-invariant it is clear that $\nu_{F_{1,2}(\mathbb{C}^3)}$ is $T^2$-invariant as well. 

\begin{prop}
\label{remark_flag}
The orbit space $\nu_{F_{1,2}(\mathbb C^3)}^{-1}(0)/T^2$ can be identified with $F_{1,2}(\mathbb{R}^3)/\mathbb{Z}_2^2$. 
\end{prop}
\begin{proof}
We can always act on $p \in \mathrm{SU}(3)/T^2$ with the two $T^2$-actions on the left and on the right to make $p_{21},p_{22},p_{23},p_{32}$ real. 
If $\nu_{F_{1,2}(\mathbb C^3)}(p)=0$, then by \eqref{multi_moment_map_flag_expression} at least one among $p_{22},p_{23},p_{32}$ is zero or $p_{33}$ is real.
In all cases we can make the second row and $p_{32},p_{33}$ real.
Further, using the orthogonality relation $p_{21}\conjugate{p}_{31}+p_{22}\conjugate{p}_{32}+p_{23}\conjugate{p}_{33}=0$ and the actions of the tori one can assume $p_{31}$ is real as well, 
so both second and third row of $p$ can be made real. Call such rows $u_2,u_3$ respectively. Since
$$F_{1,2}(\mathbb C^3) \cong \{([u_2],[u_3]) \in \mathbb{CP}^2 \times \mathbb{CP}^2: [u_2] \perp [u_3]\},$$
we find that $$\nu_{F_{1,2}(\mathbb C^3)}^{-1}(0)/T^2 = \{([u_2],[u_3]) \in \mathbb{RP}^2 \times \mathbb{RP}^2: [u_2] \perp [u_3]\}/\mathbb{Z}_2^2 \cong F_{1,2}(\mathbb R^3)/\mathbb{Z}_2^2,$$
where $\mathbb{Z}_2^2$ is the real part of $T^2$. 
\end{proof}

\begin{prop}
\label{critical_orbits_flag}
There are exactly two $T^2$-orbits of critical points where $\nu_{F_{1,2}(\mathbb{C}^3)}$ does not vanish. Thus they give maximum and minimum of $\nu_{F_{1,2}(\mathbb{C}^3)}$.
\end{prop}
\begin{remark}
\label{remark_zero_critical_points}
Based on the algebraic computations that follow, the structure of zero critical sets is not as clear as in the case of the six-sphere. 
Hence we postpone its description until the construction of the graph, ignoring all cases yielding zero critical orbits.
\end{remark}
\begin{proof}
We compute directly $d\nu_{F_{1,2}(\mathbb{C}^3)} = 3\psi_+(U,V,{}\cdot{})$, where $U,V$ are now shorthands for the vectors \eqref{generator_one_action_flag}, \eqref{generator_two_action_flag}. By Proposition \ref{nk_structure_flag} the one-form $\psi_+(U,V,{}\cdot{})$ at $p$ turns out to be
\begin{align*}
\psi_+(U,V,{}\cdot{})_{|p} & = \re (z^3\conjugate{w}^2-z^2\conjugate{w}^3)e^1 + \im (z^3\conjugate{w}^2+z^2\conjugate{w}^3)Je^1 \\
& \qquad + \im (z^3w^1-z^1w^3) e^3 +\re (z^1w^3-z^3w^1)Je^3 \\
& \qquad + \re(\conjugate{z}^2w^1-\conjugate{z}^1w^2)e^5 + \im (\conjugate{z}^2w^1+\conjugate{z}^1w^2)Je^5.
\end{align*}
This implies the point $p \in \mathrm{SU}(3)/T^2$ is critical if and only if 
$$
z^2\conjugate{w}^3 = \conjugate{z}^3w^2, \qquad
z^1w^3=z^3w^1, \qquad 
\conjugate{z}^1w^2 = z^2\conjugate{w}^1,
$$
namely when
\begin{equation}
\label{critical_system_flag}
\begin{cases}
p_{22}\conjugate{p}_{23}\conjugate{p}_{31}p_{33} = \conjugate{p}_{21}p_{23}p_{32}\conjugate{p}_{33}, \\
\conjugate{p}_{22}p_{23}\conjugate{p}_{31}p_{32} = \conjugate{p}_{21}p_{22}\conjugate{p}_{32}p_{33}, \\
\conjugate{p}_{21}p_{23}p_{31}\conjugate{p}_{32} = p_{21}\conjugate{p}_{22}\conjugate{p}_{31}p_{33}.
\end{cases}
\end{equation}
As in the proof of Proposition \ref{remark_flag} we make $p_{21},p_{22},p_{23},p_{32}$ real and non-negative. Write $a=p_{21},b=p_{22},c=p_{23}$, and $d=p_{32}$. Then \eqref{critical_system_flag} is equivalent to the equations
\begin{equation*}
bc\conjugate{p}_{31}p_{33} = ac\conjugate{p}_{33}d, \qquad bcd\conjugate{p}_{31} = abp_{33}d, \qquad acp_{31}d = ab\conjugate{p}_{31}p_{33}.
\end{equation*}
Our set-up is invariant under cyclic permutations of columns or rows of $p$ up to a sign of $\nu_{F_{1,2}(\mathbb{C}^3)}$, so in order to work out stationary orbits we can distinguish the cases $c \neq 0$ and at least one between $a$ and $b$ is zero, or $a,b,c \neq 0$. 

In the first case the system is easy to discuss and generates critical points where the multi-moment map vanishes. In the second case $d$ cannot be $0$, otherwise the criticality conditions would imply either $p_{31}=0$ or $p_{33}=0$, namely $a=0$ or $c=0$. Then our equations are
\begin{equation*}
b\conjugate{p}_{31}p_{33} = a\conjugate{p}_{33}d, \qquad cp_{31} = a\conjugate{p}_{33}, \qquad cp_{31}d = b\conjugate{p}_{31}p_{33}.
\end{equation*}
Set $p_{31} \coloneqq \rho e^{i\vartheta}, p_{33} \coloneqq \sigma e^{i\varphi}$, so that the system becomes
\begin{equation}
\label{arguments}
b\rho \sigma e^{i(\varphi-\vartheta)} = a\sigma e^{-i\varphi}d, \qquad c\rho e^{i\vartheta} = a\sigma e^{-i\varphi}, \qquad c\rho e^{i\vartheta}d = b \rho \sigma e^{i(\varphi-\vartheta)}.
\end{equation}
Observe that $\rho = 0$ if and only if $\sigma = 0$, so $d=1$ and $b=0$, contradiction. So $p_{31},p_{33} \neq 0$ and a comparison of the arguments in \eqref{arguments} shows that $3\varphi \equiv 0, \vartheta \equiv -\varphi \pmod{2\pi}$. Comparing the radii we obtain $ad=b\rho, c\rho = a\sigma, cd = b\sigma$, so $\rho = ad/b$, $\sigma = cd/b$ and $p_{31} = (ad/b)e^{-i\varphi}, p_{33} = (cd/b)e^{i\varphi}$. Now second and third row of $p$ have unit length, whence $a^2d^2/b^2+d^2+c^2d^2/b^2 =\nolinebreak 1$, which implies $d=b$. Our matrix has then the form
$$
p=
\left(
\begin{matrix}
p_{11} & p_{12} & p_{13} \\
a & b & c \\
ae^{-i\varphi} & b & ce^{i\varphi}
\end{matrix}
\right),
$$
and the constraint $p \in \mathrm{SU}(3)$ gives
\begin{alignat*}{2}
&\lvert p_{11}\rvert^2 + 2a^2 = 1, \qquad &&\conjugate{p}_{11}p_{12}=-ab(1+e^{i\varphi}), \\
&\lvert p_{12}\rvert^2 + 2b^2 = 1, \qquad &&p_{11}\conjugate{p}_{13}=-ac(1+e^{i\varphi}), \\
&\lvert p_{13}\rvert^2 + 2c^2 = 1, \qquad &&\conjugate{p}_{12}p_{13}=-bc(1+e^{i\varphi}).
\end{alignat*}
The second column in particular implies the chain of equalities
\begin{equation*}
\frac{\conjugate{p}_{11}p_{12}}{ab} =  \frac{p_{11}\conjugate{p}_{13}}{ac} = \frac{\conjugate{p}_{12}p_{13}}{bc} = -1-e^{i\varphi},
\end{equation*}
whereas the first one allows to write
$p_{11} = \sqrt{1-2a^2} e^{i\alpha}, p_{12} = \sqrt{1-2b^2}e^{i\beta}, p_{13} = \sqrt{1-2c^2}e^{i\gamma}$. We end up with three possibilities: $\varphi = 0, \varphi = 2\pi/3, \varphi = 4\pi/3$. In the first one, two rows of the matrix $p$ are the same, so the determinant vanishes. We can then assume $\varphi = 2\pi/3$, so that 
\begin{equation*}
\frac{\conjugate{p}_{11}p_{12}}{ab} =  \frac{p_{11}\conjugate{p}_{13}}{ac} = \frac{\conjugate{p}_{12}p_{13}}{bc} = e^{4\pi i/3}.
\end{equation*}
Comparing the arguments we find $\beta \equiv \alpha+4\pi/3 \pmod{2\pi}$ and $\gamma \equiv \alpha + 2\pi/3 \pmod{2\pi}$. Comparing the radii instead we obtain $a=b=c = 1/\sqrt{3}$. Imposing the condition $\det p = 1$ one gets $\alpha \equiv 7\pi/6 \pmod{2\pi}$, so 
\begin{equation}
\label{critical_point_one_flag}
p = \frac{1}{\sqrt{3}}\left(
\begin{matrix}
i\omega & i & i\omega^2 \\
1 & 1 & 1 \\
\omega^2 & 1 & \omega
\end{matrix}
\right), \quad \text{ with $\omega = e^{2\pi i/3}$.}
\end{equation}
This gives a $T^2$-orbit of points of minimum, the value of the multi-moment map at $p$ is $-\sqrt{3}/2$. The last case $\varphi = 4\pi /3$ can be discussed similarly, and the point we find turns out to be
\begin{equation*}
\conjugate{p} = \frac{1}{\sqrt{3}}\left(
\begin{matrix}
-i\omega^2 & -i & -i\omega \\
1 & 1 & 1 \\
\omega & 1 & \omega^2
\end{matrix}
\right), \quad \text{ with $\omega = e^{2\pi i/3}$.}
\end{equation*}
Since $\nu_{F_{1,2}(\mathbb{C}^3)}(\conjugate{p}) = -\nu_{F_{1,2}(\mathbb{C}^3)}(p)$, by \eqref{critical_point_one_flag} the value of the multi-moment map is $\sqrt{3}/2$. Summing up, we got two stationary orbits giving extrema, and $\im \nu_{F_{1,2}(\mathbb{C}^3)} = [-\sqrt{3}/2,\sqrt{3}/2]$.
\end{proof}

The goal now is to find which pairs of subspaces $(L,U)$ are fixed by some non-trivial element $T^2$ and compute their stabilisers. It will turn out that the action is not effective, a copy of $\mathbb{Z}_3$ in $T^2$ fixes all the flags. However, there is an isomorphism $T^2 \cong T^2/\mathbb{Z}_3$ given by $(e^{i\vartheta},e^{i\phi}) \mapsto (e^{3i\vartheta},e^{i(\vartheta-\phi)})$. In the case below where $\mathbb{Z}_3$ appears as a discrete stabilizer of all the flags, we can use this isomorphism to argue that the action of $T^2/\mathbb{Z}_3 \cong T^2$ is effective and the discrete stabilizers are all trivial.

Take a non-zero $z=(z^1,z^2,z^3) \in \mathbb{C}^3$ and assume that $L \coloneqq \Span(z)$ is a $T^2$-invariant one-dimensional subspace of $\mathbb{C}^3$. The equation we want to solve is  $A_{\vartheta,\phi}z = \lambda(\vartheta,\phi)z$, with $\lambda$ some complex-valued function of $\vartheta,\phi$. Explicitly 
\begin{equation*}
e^{i\vartheta}z^1 = \lambda z^1, \qquad e^{i\phi}z^2 = \lambda z^2, \qquad e^{-i(\vartheta+\phi)} z^3 = \lambda z^3.
\end{equation*} 

The case $e^{i\vartheta}=\lambda, e^{i\phi} = e^{i\vartheta}$ gives two subcases, either $3\vartheta \equiv 0 \pmod{2\pi}$ or $z^3 = 0$. In the former we have $\vartheta \in \{0,2\pi/3,4\pi/3\} \pmod{2\pi}$ and $z^i \neq 0$ for every $i = 1,2,3$, which gives a discrete stabilizer of $L$ as $\vartheta \equiv \phi \pmod{2\pi}$. This is a copy of $\mathbb{Z}_3$ and we can argue as above to conclude that the stabilizer is trivial. Denote by $F_1,F_2,F_3$ the standard basis of $\mathbb{C}^3$. A plain discussion of the remaining cases yields the solutions
$$
\begin{cases}
\mathbb{C}F_1,\mathbb{C}F_2,\mathbb{C}F_3, & \text{fixed by all of } T^2, \\
\mathbb{C}F_1\oplus \mathbb{C}F_2, \mathbb{C}F_1\oplus \mathbb{C}F_3, \mathbb{C}F_2\oplus \mathbb{C}F_3, & \text{fixed by a circle}.
\end{cases}
$$
Since the $T^2$-action preserves $\mathbb{C} F_i$ and the angles between two vectors, we have that $\mathbb{C} F_j \oplus \mathbb{C} F_k$ is preserved by $T^2$ as well, for different $i,j,k \in \{1,2,3\}$. On the other hand, since $\mathbb{S}^1$ preserves $\mathbb{C}z$, $z \in \Span\{F_i,F_j\}, i \neq j$, the pairs $(\mathbb{C}z, \mathbb{C}F_i \oplus \mathbb{C}F_j)$ and $(\mathbb{C}F_i, \mathbb{C}F_j \oplus \Span\{z\})$, are fixed by $\mathbb{S}^1$. Therefore, we have six points in $F_{1,2}(\mathbb{C}^3)$ fixed by all of $T^2$ and nine edges corresponding to two-dimensional subspaces of points fixed by $\mathbb{S}^1$. The six points are represented by the $A_{\alpha,\beta \gamma}$:
\begin{align*}
A_{1,12} = (\mathbb{C} F_1, \mathbb{C} F_1 \oplus \mathbb{C} F_2), \qquad & A_{2,23} = (\mathbb{C} F_2, \mathbb{C} F_2 \oplus \mathbb{C} F_3), \\
A_{1,13} = (\mathbb{C} F_1, \mathbb{C} F_1 \oplus \mathbb{C} F_3), \qquad & A_{3,13} = (\mathbb{C} F_3, \mathbb{C} F_1 \oplus \mathbb{C} F_3), \\
A_{2,12} = (\mathbb{C} F_2, \mathbb{C} F_1 \oplus \mathbb{C} F_2), \qquad & A_{3,23} = (\mathbb{C} F_3, \mathbb{C} F_2 \oplus \mathbb{C} F_3).
\end{align*}
The edges $a_i, i=1\dots,9$ are the following:
\begin{alignat*}{3}
& z_1 \in \Span\{F_1,F_2\}, \quad && z_2 \in \Span\{F_1,F_3\}, \quad && z_3 \in \Span\{F_2,F_3\},\\
& a_1 = (\mathbb{C} z_1, \mathbb{C} F_1 \oplus \mathbb{C} F_2), \quad && a_4 = (\mathbb{C} z_2, \mathbb{C} F_1 \oplus \mathbb{C} F_3), \quad && a_7 = (\mathbb{C} F_1, \mathbb{C} F_1 \oplus \mathbb{C} z_3),\\
& a_2 = (\mathbb{C} z_1, \mathbb{C} F_3 \oplus \mathbb{C} z_1), \quad && a_5 = (\mathbb{C} F_2, \mathbb{C} z_2 \oplus \mathbb{C} F_2), \quad && a_8 = (\mathbb{C} z_3, \mathbb{C} F_2 \oplus \mathbb{C} F_3),\\
& a_3 = (\mathbb{C} F_3, \mathbb{C} z_1 \oplus \mathbb{C} F_3), \quad && a_6 = (\mathbb{C} z_2, \mathbb{C} F_2 \oplus \mathbb{C} z_2), \quad && a_9 = (\mathbb{C} z_3, \mathbb{C} F_1 \oplus \mathbb{C} z_3).
\end{alignat*}
In order to figure out what the vertices of, say, $a_1$ are, one can take the limit $z \rightarrow F_1$ (resp.\  $z \rightarrow F_2$) and see that $a_1 \rightarrow A_{1,12}$ (resp.\  $a_1 \rightarrow A_{2,12}$), see Figure \ref{fig:flag}.
As for the case of the six-sphere above, the graph in Figure \ref{fig:flag} is the set of points where $\mathbb Z_2^2$ as in Proposition \ref{remark_flag} acts non-freely.

\begin{figure}
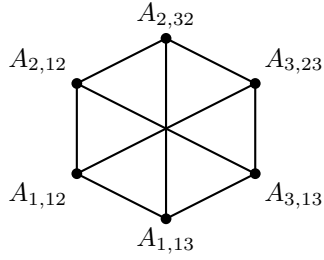

  \centering
  \tikzpicture
  [scale=1.2]
    \coordinate (A) at (0,0);
    \coordinate (B) at (0.98803162,-0.5);
    \coordinate (C) at (1.976,0);
    \coordinate (D) at (1.976,1);
    \coordinate (E) at (0.98803162,1.5);
    \coordinate (F) at (0,1);
    \draw[thick][black] (A) -- (B) ;
    \draw[thick][black] (B) -- (C); 
    \draw[thick][black] (C) -- (D);
    \draw[thick][black] (D) -- (E);
    \draw[thick][black] (E) -- (F);
    \draw[thick][black] (F) -- (A);
    \draw[thick][black] (A) -- (D);
    \draw[thick][black] (B) -- (E);
    \draw[thick][black] (C) -- (F);
     \fill (A) circle [radius=1.7pt] node[below left] {$A_{1,12}$};
    \fill (B) circle [radius=1.7pt] node[below] {$A_{1,13}$};
    \fill (C) circle [radius=1.7pt] node[below right] {$A_{3,13}$};
    \fill (D) circle [radius=1.7pt] node[above right] {$A_{3,23}$};
    \fill (E) circle [radius=1.7pt] node[above] {$A_{2,32}$};
    \fill (F) circle [radius=1.7pt] node[above left] {$A_{2,12}$};
  \endtikzpicture
  \caption{Graph for the flag manifold in $F_{1,2}(\mathbb{C}^3)/T^2$}
  \label{fig:flag}
\end{figure}

\section{The complex projective space}
\label{complex_projective_space}

The compact symplectic group $\mathrm{Sp}(2)$ acts on $\mathbb{C}^4 \cong \mathbb{H}^2$, transitively on the projective space $\mathbb{CP}^3$. An element in $p \in \mathrm{Sp}(2)$, $p = \left(\begin{smallmatrix} p_{11} & p_{12} \\ p_{21} & p_{22} \end{smallmatrix}\right)$, fixes $a\coloneqq[1:0:0:0] \in \mathbb{CP}^3$ when $p_{11}$ is a combination of $1,i$ and the quaternions $p_{12},p_{21}$ vanish. Since $p_{11}$ has unit length it must lie in a circle $\mathrm{U}(1)$. The isotropy group of $a$ is then isomorphic to $\mathrm{Sp}(1)\mathrm{U}(1)$, and $\mathbb{CP}^3$ is diffeomorphic to $\mathrm{Sp}(2)/\mathrm{Sp}(1)\mathrm{U}(1)$. Write $H \coloneqq \mathrm{Sp}(1)\mathrm{U}(1)$ and $G \coloneqq \mathrm{Sp}(2)$. We thus identify $H$ with a subgroup of $G$ containing elements of the form $\diag(e^{i\vartheta},\alpha)$, where $\alpha$ is a unit quaternion and $\vartheta$ an angle. We denote by $\mathfrak{g}$ and $\mathfrak{h}$ the Lie algebras of $G$ and $H$ respectively, so $\mathfrak{g}$ splits as $\mathfrak{h} \oplus \mathfrak{m}$.

On the Lie algebra $\mathfrak{g}$ we define the Killing form as $g_0(X,Y) \coloneqq \tr (^t\conjugate{X}Y) = -\tr (XY)$. This can be translated to any point $p$ yielding an inner product $g_p \coloneqq \re((p_{\id}^{-1})^*g_0)$ on every tangent space $T_pG$, and descends to the quotient modulo $H$ as it is bi-invariant. The construction of the almost complex structure $J$ follows from the existence of a diffeomorphism of order three as in the case of $F_{1,2}(\mathbb{C}^3)$. Consider $A \coloneqq \diag(e^{2\pi i/3},1)$ in $G$ and define an almost complex structure $J_p$ at the point $p$ by translating the endomorphism $J_0\colon \mathfrak{m} \to \mathfrak{m}, J_0X \coloneqq (2/\sqrt{3})(AXA^{-1}+\tfrac{1}{2}X)$. 
The Lie algebra $\mathfrak{g}$ is spanned by 
\begin{alignat*}{4}
E_0 & = 
\left(
\begin{matrix}
k & 0 \\
0 & 0 
\end{matrix}
\right), && \quad 
E_1 = 
\left(
\begin{matrix}
j & 0 \\
0 & 0 
\end{matrix}
\right), 
&& \quad 
E_2 = 
\tfrac{1}{\sqrt{2}}\left(
\begin{matrix}
0 & 1 \\
-1 & 0 
\end{matrix}
\right), && \quad
E_3 = 
\tfrac{1}{\sqrt{2}}\left(
\begin{matrix}
0 & i \\
i & 0 
\end{matrix}
\right), \\
E_4 & = 
\tfrac{1}{\sqrt{2}}\left(
\begin{matrix}
0 & j \\
j & 0 
\end{matrix}
\right), && \quad 
E_5 = 
\tfrac{1}{\sqrt{2}}\left(
\begin{matrix}
0 & k \\
k & 0 
\end{matrix}
\right), && \quad
E_6 = 
\left(
\begin{matrix}
i & 0 \\
0 & 0 
\end{matrix}
\right), && \quad 
E_7 = 
\left(
\begin{matrix}
0 & 0 \\
0 & i 
\end{matrix}
\right), \\
E_8 & = 
\left(
\begin{matrix}
0 & 0 \\
0 & j 
\end{matrix}
\right), && \quad
E_9 = 
\left(
\begin{matrix}
0 & 0 \\
0 & k 
\end{matrix}
\right).
\end{alignat*}
One can check the metric and the almost complex structure have the familiar shapes as in \eqref{metric_flag}, \eqref{almost_complex_structure_flag_origin}.
\begin{prop}
\label{nk_structure_projective_space}
The forms on $\mathfrak{g}=\mathfrak{sp}(2)$ given by
\begin{gather*}
\sigma_0 \coloneqq g_0(J_0{}\cdot{},{}\cdot{}) = e^{01} + e^{23} + e^{45}, \\
\varphi_0 \coloneqq e^{024}-e^{134}-e^{035}-e^{125}, \qquad \psi_0 \coloneqq e^{025}-e^{135}+e^{034}+e^{124},
\end{gather*}
descend to the quotient $\mathfrak{g}/\mathfrak{h}$ and satisfy $d\sigma_0 = 3\varphi_0$ and $d\psi_0 = -2\sigma_0 \wedge \sigma_0$. Consequently, the differential forms on $\mathrm{Sp}(2)/\mathrm{Sp}(1)\mathrm{U}(1)$ given by
$\sigma_p \coloneqq g_p(J_p{}\cdot{},{}\cdot{}), \psi_{+|p} \coloneqq \varphi_0(p^{-1}{}\cdot{},p^{-1}{}\cdot{},p^{-1}{}\cdot{})$, and $\psi_{-|p} \coloneqq \psi_0(p^{-1}{}\cdot{},p^{-1}{}\cdot{},p^{-1}{}\cdot{})$ define a nearly K\"ahler structure on $\mathbb{CP}^3$.
\end{prop}
\begin{proof}
This follows from the expressions of the differentials of $e^k,k=0,\dots,5$:
\begin{alignat*}{2}
de^0& = 2e^{16}-e^{25}-e^{34}, \qquad && de^1 = -2e^{06}-e^{24}+e^{35}, \\
de^2 & = e^{05}+e^{14}-e^{36}+e^{37}+e^{48}+e^{59}, \qquad && de^3 = e^{04}-e^{15}+e^{26}-e^{27}-e^{49}+e^{58}, \\
de^4 & = -e^{03}-e^{12}-e^{28}+e^{39}-e^{56}-e^{57}, \qquad && de^5 = -e^{02}+e^{13}-e^{38}+e^{46}+e^{47}-e^{29}.
\end{alignat*}
A routine computation gives the result.
\end{proof}

Let $T^2$ be the maximal torus in $\mathrm{Sp}(2)$ given by matrices of the form $\diag(e^{i\vartheta},e^{i\varphi})$. Two generators of $\mathfrak{t}^2$ are $U = \diag(i,0)$ and $V = \diag(0,i)$. We want to compute the vectors $p^{-1}Up$ and $p^{-1}Vp$ in terms of the coefficients of $p$ as an element of $\mathrm{Sp}(2)$. We split $p$ as $p_{ik} = p_{ik}^1+p_{ik}^2j$, where $p^1,p^2$ are now $2 \times 2$ complex matrices. The condition $p^{-1} = {^t\conjugate{p}}$ is equivalent to ${^t}\conjugate{p}^1p^1+{^t}p^2\conjugate{p}^2 = \id, {^t}\conjugate{p}^1p^2-{^t}p^2\conjugate{p}^1 = 0$. Expanding and adding the condition $p^{-1} \in \mathrm{Sp}(2)$ we get the equations
\begin{alignat*}{2}
& \lvert p_{11}^1 \rvert^2 +\lvert p_{21}^1\rvert^2 + \lvert p_{11}^2\rvert^2 + \lvert p_{21}^2\rvert^2 = 1, \qquad && \conjugate{p}_{11}^1p_{12}^1 + \conjugate{p}_{21}^1p_{22}^1 +\conjugate{p}_{12}^2p_{11}^2+\conjugate{p}_{22}^2p_{21}^2 = 0, \\
& \lvert p_{12}^1\rvert^2 + \lvert p_{22}^1\rvert^2+\lvert p_{12}^2\rvert^2 + \lvert p_{22}^2\rvert^2 = 1, \qquad && \conjugate{p}_{11}^1p_{12}^2+\conjugate{p}_{21}^1p_{22}^2-\conjugate{p}_{12}^1p_{11}^2-\conjugate{p}_{22}^1p_{21}^2 = 0, \\
& \lvert p_{11}^1\rvert^2 + \lvert p_{12}^1\rvert^2+\lvert p_{11}^2\rvert^2 + \lvert p_{12}^2\rvert^2 = 1, \qquad && p_{11}^1\conjugate{p}_{21}^1+p_{12}^1\conjugate{p}_{22}^1 + \conjugate{p}_{21}^2p_{11}^2 + \conjugate{p}_{22}^2p_{12}^2=0, \\
& \lvert p_{21}^1\rvert^2 + \lvert p_{22}^1\rvert^2+\lvert p_{21}^2\rvert^2 + \lvert p_{22}^2\rvert^2 = 1, \qquad && p_{11}^1p_{21}^2+p_{12}^1p_{22}^2-p_{21}^1p_{11}^2-p_{22}^1p_{12}^2 = 0.
\end{alignat*}
We can thus calculate the vectors generating the action:
\begin{align*}
p^{-1}Up & = \left(\begin{matrix} i(\lvert p_{11}^1\rvert^2- \lvert p_{11}^2 \rvert^2) & i(\conjugate{p}_{11}^1p_{12}^1-p_{11}^2\conjugate{p}_{12}^2) \\ i(\conjugate{p}_{12}^1p_{11}^1-p_{12}^2\conjugate{p}_{11}^2) & i(\lvert p_{12}^1\rvert^2-\lvert p_{12}^2 \rvert^2) \end{matrix}\right) \\
& \qquad +\left(\begin{matrix} 2i\conjugate{p}_{11}^1p_{11}^2 & i(\conjugate{p}_{11}^1p_{12}^2+p_{11}^2\conjugate{p}_{12}^1) \\ i(\conjugate{p}_{12}^1p_{11}^2+p_{12}^2\conjugate{p}_{11}^1) & 2ip_{12}^2\conjugate{p}_{12}^1\end{matrix}\right)j, \\[10pt]
p^{-1}Vp & = \left(\begin{matrix} i(\lvert p_{21}^1\rvert^2-\lvert p_{21}^2\rvert^2) & i(\conjugate{p}_{21}^1p_{22}^1-p_{21}^2\conjugate{p}_{22}^2) \\ i(\conjugate{p}_{22}^1p_{21}^1-p_{22}^2\conjugate{p}_{21}^2) & i(\lvert p_{22}^1\rvert^2-\lvert p_{22}^2 \rvert^2)\end{matrix} \right) \\ 
& \qquad + \left(\begin{matrix} 2i\conjugate{p}_{21}^1p_{21}^2 & i(\conjugate{p}_{21}^1p_{22}^2+p_{21}^2\conjugate{p}_{22}^1) \\ i(\conjugate{p}_{22}^1p_{21}^2+p_{22}^2\conjugate{p}_{21}^1) & 2i\conjugate{p}_{22}^1p_{22}^2 \end{matrix} \right)j.
\end{align*}
The Lie algebra $\mathfrak{h}$ contains elements of the form $\left(\begin{smallmatrix} ia & 0 \\ 0 & ib\end{smallmatrix}\right)+\left(\begin{smallmatrix} 0 & 0 \\ 0 & c\end{smallmatrix}\right)j$, with $a,b$ real and $c$ complex. The projections $(p^{-1}Up)_{\mathfrak{m}}, (p^{-1}Vp)_{\mathfrak{m}}$ must be of the form $\bigl(\begin{smallmatrix} xj & \rho \\ -\conjugate{\rho} & 0\end{smallmatrix}\bigr)$, for $x$ a complex number and $\rho$ a quaternion, \nolinebreak so
\begin{align*}
(p^{-1}Up)_{\mathfrak{m}} &= \left(\begin{matrix} 0 & \alpha \\ -\conjugate{\alpha} & 0 \end{matrix}\right) +\left(\begin{matrix} \gamma j & 0 \\ 0 & 0\end{matrix}\right), \quad \begin{cases} \alpha = i(\conjugate{p}_{11}^1p_{12}^1-p_{11}^2\conjugate{p}_{12}^2)+i(\conjugate{p}_{11}^1p_{12}^2+p_{11}^2\conjugate{p}_{12}^1)j, \\ \gamma = 2i\conjugate{p}_{11}^1p_{11}^2, \end{cases} \\
(p^{-1}Vp)_{\mathfrak{m}} &= \left(\begin{matrix} 0 & \beta \\ -\conjugate{\beta} & 0 \end{matrix}\right) +\left(\begin{matrix} \delta j & 0 \\ 0 & 0\end{matrix}\right), \quad \begin{cases}\beta = i(\conjugate{p}_{21}^1p_{22}^1-p_{21}^2\conjugate{p}_{22}^2)+i(\conjugate{p}_{21}^1p_{22}^2+p_{21}^2\conjugate{p}_{22}^1)j, \\ \delta = 2i\conjugate{p}_{21}^1p_{21}^2.\end{cases}
\end{align*}
We now write $(p^{-1}Up)_{\mathfrak{m}}, (p^{-1}Vp)_{\mathfrak{m}}$ in terms of the basis introduced above. Then
\begin{align*}
(p^{-1}Up)_{\mathfrak{m}} & = 2\re (\conjugate{p}_{11}^1p_{11}^2) E_0 - 2\im (\conjugate{p}_{11}^1p_{11}^2)  E_1 \\
& \qquad + \sqrt{2}\Bigl(-\im (\conjugate{p}_{11}^1p_{12}^1-p_{11}^2\conjugate{p}_{12}^2) E_2+\re (\conjugate{p}_{11}^1p_{12}^1-p_{11}^2\conjugate{p}_{12}^2)E_3\\
& \qquad \qquad \qquad -\im (\conjugate{p}_{11}^1p_{12}^2+p_{11}^2\conjugate{p}_{12}^1)E_4 + \re (\conjugate{p}_{11}^1p_{12}^2+p_{11}^2\conjugate{p}_{12}^1)E_5\Bigr), \\
(p^{-1}Vp)_{\mathfrak{m}} & = 2\re (\conjugate{p}_{21}^1p_{21}^2) E_0 - 2\im (\conjugate{p}_{21}^1p_{21}^2)  E_1 \\
& \qquad + \sqrt{2}\Bigl(-\im (\conjugate{p}_{21}^1p_{22}^1-p_{21}^2\conjugate{p}_{22}^2) E_2+\re (\conjugate{p}_{21}^1p_{22}^1-p_{21}^2\conjugate{p}_{22}^2)E_3\\
& \qquad \qquad \qquad -\im (\conjugate{p}_{21}^1p_{22}^2+p_{21}^2\conjugate{p}_{22}^1)E_4 + \re (\conjugate{p}_{21}^1p_{22}^2+p_{21}^2\conjugate{p}_{22}^1)E_5\Bigr).
\end{align*}
It is convenient to write
\begin{align*}
(p^{-1}Up)_{\mathfrak{m}} & = f'E_0+e'E_1+\sqrt{2}(a'E_2+b'E_3+c'E_4+d'E_5), \\
(p^{-1}Vp)_{\mathfrak{m}} & = f''E_0+e''E_1+\sqrt{2}(a''E_2+b''E_3+c''E_4+d''E_5),
\end{align*}
where $\alpha = a'+b'i+c'j+d'k, \beta = a''+b''i+c''j+d''k, \gamma = e'+f'i, \delta = e''+f''i$. 
Using the conditions $p, p^{-1} \in \mathrm{Sp}(2)$, the multi-moment map $\nu_{\mathbb{CP}^3}(p) = \sigma((p^{-1}Up)_{\mathfrak{m}},(p^{-1}Vp)_{\mathfrak{m}})$ gets the form
\begin{align}
\label{multi_moment_map_complex_projective_space_expression}
\nu_{\mathbb{CP}^3}(p) & = \im (\gamma \conjugate{\delta}) + 2\re (i\alpha \conjugate{\beta}) \nonumber \\
& = 4\im (\conjugate{p}_{11}^1p_{11}^2p_{21}^1\conjugate{p}_{21}^2) - 2\im \bigl((\conjugate{p}_{11}^1p_{12}^1-p_{11}^2\conjugate{p}_{12}^2)(p_{21}^1\conjugate{p}_{22}^1-\conjugate{p}_{21}^2p_{22}^2) \nonumber \\
& \qquad \qquad \qquad \qquad \qquad \qquad \quad + (\conjugate{p}_{11}^1p_{12}^2+p_{11}^2\conjugate{p}_{12}^1)(p_{21}^1\conjugate{p}_{22}^2+\conjugate{p}_{21}^2p_{22}^1)\bigr) \nonumber \\
& = 4\im (\conjugate{p}_{11}^1p_{11}^2p_{21}^1\conjugate{p}_{21}^2) - 2\im \bigl((\conjugate{p}_{11}^1p_{12}^1-p_{11}^2\conjugate{p}_{12}^2)(2p_{21}^1\conjugate{p}_{22}^1+p_{11}^1\conjugate{p}_{12}^1+p_{12}^2\conjugate{p}_{11}^2) \nonumber \\
& \qquad \qquad \qquad \qquad \qquad \qquad \quad + (\conjugate{p}_{11}^1p_{12}^2+p_{11}^2\conjugate{p}_{12}^1)(2p_{21}^1\conjugate{p}_{22}^2+p_{11}^1\conjugate{p}_{12}^2-p_{12}^1\conjugate{p}_{11}^2)\bigr) \nonumber \\
& = 4\im (\conjugate{p}_{11}^1p_{11}^2p_{21}^1\conjugate{p}_{21}^2) - 4\im \bigl(\conjugate{p}_{11}^1p_{21}^1(p_{12}^1\conjugate{p}_{22}^1+p_{12}^2\conjugate{p}_{22}^2)+p_{11}^2p_{21}^1(\conjugate{p}_{12}^1\conjugate{p}_{22}^2-\conjugate{p}_{12}^2\conjugate{p}_{22}^1)\bigr) \nonumber \\
& = 4\im (\conjugate{p}_{11}^1p_{11}^2p_{21}^1\conjugate{p}_{21}^2) - 4\im \bigl(-\conjugate{p}_{11}^1p_{21}^1(p_{11}^1\conjugate{p}_{21}^1+p_{11}^2\conjugate{p}_{21}^2)+p_{11}^2p_{21}^1(\conjugate{p}_{11}^2\conjugate{p}_{21}^1-\conjugate{p}_{11}^1\conjugate{p}_{21}^2)\bigr) \nonumber \\
& = 12\im (\conjugate{p}_{11}^1p_{11}^2p_{21}^1\conjugate{p}_{21}^2) \nonumber \\
& = 3\im (\gamma \conjugate{\delta}),
\end{align}
which is indeed invariant under the torus action, because $\gamma,\delta$ are invariant. 

\begin{prop}
\label{complex_projective_space_zero_set}
The orbit space $\nu_{\mathbb{CP}^3}^{-1}(0)/T^2$ can be identified with $\mathbb{RP}^3/\mathbb{Z}_2^2$.
\end{prop}
\begin{proof}
Combine the action of $T^2$ on the left and $\mathrm{Sp}(1)\mathrm{U}(1)$ on the right on $p \in \mathrm{Sp}(2)/\mathrm{Sp}(1)\mathrm{U}(1)$ to obtain $p_{12},p_{11}^1,p_{11}^2,p_{21}^1$ real.
The orthogonality relations and expression \eqref{multi_moment_map_complex_projective_space_expression} imply that $\nu_{\mathbb{CP}^3}(p)=0$ when
$p_{21}^2$ is real as well. But by definition of the action of $\mathrm{Sp}(2)$ on $\mathbb{CP}^3$, $p \in \mathrm{Sp}(2)/\mathrm{Sp}(1)\mathrm{U}(1)$ corresponds to
$[p_{11}^1:p_{21}^1:p_{11}^2:p_{21}^2] \in \mathbb{CP}^3$. Thus $\nu_{\mathbb{CP}^3}^{-1}(0)/T^2 = \mathbb{RP}^3/\mathbb{Z}_2^2$, where $\mathbb{Z}_2^2$ is the real part of \nolinebreak $T^2$.
\end{proof}

\begin{prop}
There are exactly two $T^2$-orbits of critical points where the multi-moment map $\nu_{\mathbb{CP}^3}$ does not vanish. Thus they give maximum and minimum of $\nu_{\mathbb{CP}^3}$.
\end{prop}
\begin{proof}
Remark \ref{remark_zero_critical_points} applies to this case as well, so we concentrate on non-zero critical orbits. We want to find the points where $\psi_+(U,V,{}\cdot{}) = 0$. Again, we use $U$ and $V$ as shorthands for $(p^{-1}Up)_{\mathfrak{m}}, (p^{-1}Vp)_{\mathfrak{m}}$. Therefore
\begin{align*}
V \chair U \chair (e^0 \wedge e^2 \wedge e^4) & = 2(a'c''-a''c')e^0 + \sqrt{2}(c'f''-f'c'')e^2 + \sqrt{2}(f'a''-a'f'')e^4, \\
-V \chair U \chair (Je^0 \wedge Je^2 \wedge e^4) & = 2(c'b''-b'c'')Je^0 + \sqrt{2}(e'c''-c'e'')Je^2 + \sqrt{2}(b'e''-b''e')e^4, \\
-V \chair U \chair (e^0 \wedge Je^2 \wedge Je^4) & = 2(b''d'-b'd'')e^0 + \sqrt{2}(f'd''-d'f'')Je^2+\sqrt{2}(b'f''-f'b'')Je^4, \\
-V \chair U \chair (Je^0 \wedge e^2 \wedge Je^4) & = 2(a''d'-a'd'')Je^0 + \sqrt{2}(e'd''-d'e'')e^2 + \sqrt{2}(a'e''-e'a'')Je^4.
\end{align*}
The equation $\psi_+(U,V,{}\cdot{})_{|p}=0$ is then equivalent to
\begin{alignat*}{2}
& (a'c''-a''c')+(b''d'-b'd'')=0, \qquad && (e'c''-c'e'')+(f'd''-d'f'')=0, \\
& (b''c'-b'c'')+(a''d'-a'd'')=0, \qquad && (f'a''-a'f'')+(b'e''-b''e')=0, \\
& (c'f''-f'c'')+(e'd''-d'e'')=0, \qquad && (b'f''-f'b'')+(a'e''-e'a'') = 0.
\end{alignat*}
A direct calculation shows that these conditions may be rephrased using $\alpha,\beta,\gamma,\delta$ as
\begin{equation*}
\alpha \conjugate{\beta}  \in \Span\{1,i\}, \quad \alpha \delta - \beta \gamma \in \Span\{1,i\}, \quad \alpha \conjugate{\delta}-\beta\conjugate{\gamma} \in \Span\{ j,k\}.
\end{equation*}
In terms of the $p_{ij}^k$, the latter are respectively
$$
\begin{cases}
\conjugate{p}_{11}^1\conjugate{p}_{21}^1(p_{12}^1p_{22}^2-p_{12}^2p_{22}^1)+p_{11}^2p_{21}^2(\conjugate{p}_{12}^1\conjugate{p}_{22}^2-\conjugate{p}_{12}^2\conjugate{p}_{22}^1) \\
\qquad +\conjugate{p}_{11}^1p_{21}^2(p_{12}^2\conjugate{p}_{22}^2+p_{12}^1\conjugate{p}_{22}^1)-p_{11}^2\conjugate{p}_{21}^1(\conjugate{p}_{12}^1p_{22}^1+\conjugate{p}_{12}^2p_{22}^2)=0, \\[5pt]
\conjugate{p}_{11}^1p_{11}^2(\conjugate{p}_{21}^1p_{22}^2+p_{21}^2\conjugate{p}_{22}^1)-\conjugate{p}_{21}^1p_{21}^2(\conjugate{p}_{11}^1p_{12}^2+p_{11}^2\conjugate{p}_{12}^1)=0, \\[5pt]
p_{11}^1\conjugate{p}_{11}^2(\conjugate{p}_{21}^1p_{22}^1-p_{21}^2\conjugate{p}_{22}^2)-p_{21}^1\conjugate{p}_{21}^2(\conjugate{p}_{11}^1p_{12}^1-p_{11}^2\conjugate{p}_{12}^2)=0.
\end{cases}
$$
As in the proof of Proposition \ref{complex_projective_space_zero_set} we combine the left action of $T^2$ and the right action of $\mathrm{Sp}(1)\mathrm{U}(1)$ so that $p_{12} = c$ is a non-negative real number, $p_{11}=a+bj$ for $a,b$ non-negative real numbers, and $p_{21} = d+\rho j$, where $d$ is a non-negative real and $\rho$ is complex. The system giving critical points then reduces to
\begin{equation}
\label{system}
\begin{cases}
c\bigl(a(\tau d + \rho \conjugate{\sigma}) - b(\sigma d - \rho \conjugate{\tau})\bigr) = 0, \\
b\bigl(a(\tau d+ \rho \conjugate{\sigma}) - c\conjugate{\rho}d\bigr) = 0, \\
a\bigl(b(\sigma d - \rho \conjugate{\tau}) - c\conjugate{\rho}d\bigr) = 0,
\end{cases}
\end{equation}
and the conditions defining $\mathrm{Sp}(2)$ are 
\begin{alignat*}{2}
& a^2+b^2+d^2+\lvert \rho\rvert^2 = 1, \qquad && bc-\tau d + \rho \conjugate{\sigma} = 0, \\
& ac+\sigma d+\rho \conjugate{\tau} = 0, \qquad && c^2 + \lvert \sigma \rvert^2 + \lvert \tau\rvert^2 = 1.
\end{alignat*}
We distinguish two main cases: $c=0$ and $c > 0$. The first one yields only critical points where the multi-moment
map vanishes, so we jump to the second. The only interesting subcase is when $a,b$ are both non-zero. One can easily see that the first equation in \eqref{system} is redundant. The orthogonality relations for $\mathrm{Sp}(2)$ yield $b\sigma d = a\rho \conjugate{\sigma}$. So we have either $\sigma = 0$ or $\sigma \neq 0$. We focus on the latter, hence $a,b,c,\sigma \neq 0$. The critical conditions are
\begin{equation}
\label{system_2}
\begin{cases}
a(\tau d + \rho \conjugate{\sigma}) = c\conjugate{\rho}d, \\
b(\sigma d - \rho \conjugate{\tau}) = c\conjugate{\rho}d,
\end{cases}
\end{equation}
and orthogonality of the columns of matrices in $\mathrm{Sp}(2)$ yields $\tau d - \rho \conjugate{\sigma} = bc, \sigma d + \rho \conjugate{\tau} = -ac$. Plugging the last two equations in \eqref{system_2} we find:
\begin{alignat*}{4}
& a\rho\conjugate{\sigma} = b\sigma d, \qquad && a\tau d = -b\rho \conjugate{\tau}, \qquad && abc = a\tau d - b\sigma d, \qquad &&abc = -a\rho\conjugate{\sigma} - b\rho \conjugate{\tau}.
\end{alignat*}
Write $\rho = Re^{ir}, \sigma = Se^{is}$, and $\tau = Te^{it}$. Comparing the angles in the first two equations we find the congruences $s \equiv r-s \pmod{2\pi}, t \equiv \pi+r-t \pmod{2\pi}$, which imply $t = \pi /2+s \pmod{\pi}$. This gives two subcases: $e^{it} = ie^{is}$ and $e^{it} = -ie^{is}$, which we solve in the same fashion. Observe that $e^{ir}=e^{2is}$, so plugging these results in our starting equations $a\rho\conjugate{\sigma} = b\sigma d$ and $a\tau d = -b\rho \conjugate{\tau}$ we find $R = ad/b=bd/a$, hence $a=b$. Therefore, the equations $a\rho\conjugate{\sigma} = b\sigma d$ and $a\tau d = -b\rho \conjugate{\tau}$ simplify as $\rho \conjugate{\sigma} = \sigma d, \tau d = -\rho \conjugate{\tau}$. They imply $\sigma = \rho \conjugate{\sigma}/d$ and $\tau = -\rho\conjugate{\tau}/d$, so $\lvert \sigma\rvert^2+\lvert \tau\rvert^2 = (|\rho|^2/d^2)(|\sigma|^2+|\tau|^2)$, whence $d = |\rho|$, namely $\rho = de^{ir}$. But $2a^2+c^2=a^2+b^2+c^2 = 1$ and $a^2+b^2+d^2+|\rho|^2 = 1 = 2a^2+2d^2$, so $c^2=2d^2$, that is $c = \sqrt{2}d$. Observe now that by the conditions $abc = a\tau d - b\sigma d$ and $abc = -a\rho\conjugate{\sigma} - b\rho \conjugate{\tau}$ we have $\tau-\sigma = \sqrt{2}a$ and $\sigma+\tau = -\sqrt{2}ae^{ir}$, so
\begin{equation*}
2\tau = \sqrt{2}a(1-e^{ir}), \qquad 2\sigma = -\sqrt{2}a(1+e^{ir}).
\end{equation*}
Then $bd\sigma-b\rho\conjugate{\tau} = cd\conjugate{\rho}$ becomes 
$$-\tfrac{\sqrt{2}}{2}a^2d(1+e^{ir})-\tfrac{\sqrt{2}}{2}a^2de^{ir}(1-e^{-ir}) = \sqrt{2}d^3e^{-ir},$$
that is $-a^2e^{2ir} = d^2$. Therefore $e^{2ir}$ must be real and negative, whence $r = \pm \pi/2$ and $a = d$. But since the first column has unit length, $4a^2=1$, so $a=1/2=b=d=c/\sqrt{2} = |\rho|$. We obtain the critical points 
$$
\left(
\begin{matrix}
\tfrac{1}{2}+\tfrac{1}{2}j & \tfrac{1}{\sqrt{2}} \\
\tfrac{1}{2}+\tfrac{1}{2}ij & -\tfrac{1}{2\sqrt{2}}(1+i) + \tfrac{1}{2\sqrt{2}}(1-i)j
\end{matrix}
\right), \quad 
\left(
\begin{matrix}
\tfrac{1}{2}+\tfrac{1}{2}j & \tfrac{1}{\sqrt{2}} \\
\tfrac{1}{2}-\tfrac{1}{2}ij & -\tfrac{1}{2\sqrt{2}}(1-i) + \tfrac{1}{2\sqrt{2}}(1+i)j
\end{matrix}
\right).
$$
Finally $\nu_{\mathbb{CP}^3}(\mathbb{CP}^3) = [-3/4,3/4]$, so we end up with two critical $T^2$-orbits giving extrema.
\end{proof}

The maximal two-torus in $\mathrm{SU}(4)$ is given by elements $A_{\vartheta,\phi} = \diag(e^{i\vartheta}, e^{i\phi}, e^{-i\vartheta}, e^{-i\phi})$. Explicitly on $\mathbb{C}\mathbb{P}^3$ we have 
\begin{equation*}
A_{\vartheta,\phi}([z^1 \colon z^2 \colon z^3 \colon z^4]) = [e^{i\vartheta}z^1 \colon e^{i\phi}z^2 \colon e^{-i\vartheta}z^3 \colon e^{-i\phi}z^4].
\end{equation*}
Observe that this action is not effective, because $A_{0,0},A_{\pi,\pi}$ fix all points of $\mathbb{CP}^3$, and these are the only elements of the torus doing that. The morphism $(e^{i\vartheta},e^{i\phi}) \mapsto (e^{i2\vartheta},e^{i(\phi-\vartheta)})$ from the torus to itself induces an isomorphism $T^2 \cong T^2/\mathbb{Z}_2$, so the action of $T^2/\mathbb{Z}_2 \cong T^2$ on $\mathbb{CP}^3$ is effective. We want the solutions of $A_{\vartheta,\phi}([z^1\colon z^2 \colon z^3 \colon z^4]) = [z^1 \colon z^2 \colon z^3 \colon z^4]$, for a non-trivial $A_{\vartheta,\phi}$. The homogeneous coordinates allow us to simplify the equations so as to get $z^1 = \lambda z^1, e^{i(\phi-\vartheta)}z^2 = \lambda z^2, e^{-i2\vartheta}z^3 = \lambda z^3, e^{-i(\vartheta+\phi)}z^4 = \lambda z^4$, where $\lambda = \lambda(\vartheta,\phi)$ is a complex-valued function. We distinguish the cases $\lambda = 1$ and $z^1 = 0$. In the former we find the solutions
$$
\begin{cases}
[1 \colon 0\colon 0\colon 0], & \text{ fixed by all of } T^2, \\
[z^1\colon z^2\colon 0\colon 0],[z^1\colon 0\colon z^3\colon 0], [z^1\colon 0\colon 0\colon z^4], & \text{ fixed by a circle}.
\end{cases}
$$
Integrating these results with those of the case $z^1 = 0$, we have four points fixed by all of $T^2$, namely $[1\colon 0\colon 0\colon 0], [0\colon 1\colon 0\colon 0], [0\colon 0\colon 1\colon 0], [0\colon 0\colon 0\colon 1]$, and six points fixed by $\mathbb{S}^1$, which are
\begin{align*}
[z^1\colon z^2\colon 0\colon 0], &\qquad [z^1\colon 0\colon z^3\colon 0], \qquad [z^1\colon 0\colon 0\colon z^4],\\
 [0\colon z^2\colon z^3\colon 0], & \qquad [0\colon z^2\colon 0\colon z^4], \qquad [0\colon 0\colon z^3\colon z^4].
 \end{align*}
For any point fixed by $\mathbb{S}^1$ we see that if one of the coordinates approaches $0$ then it meets one of the points fixed by all of $T^2$, see Figure \ref{fig:CP3}.
Again, the graph in Figure \ref{fig:CP3} is the set of points where $\mathbb Z_2^2$ as in Proposition \ref{complex_projective_space_zero_set} acts non-freely.

\begin{figure}
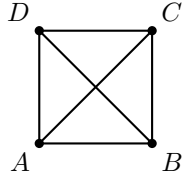

  \centering
  \tikzpicture
  [scale=0.5]
    \coordinate (A) at (0,0);
    \coordinate (B) at (3,0);
    \coordinate (C) at (3,3);
    \coordinate (D) at (0,3);
    \draw[thick][black] (A) -- (B);
    \draw[thick][black] (B) -- (C);
    \draw[thick][black] (C) -- (D);
    \draw[thick][black] (D) -- (A);
    \draw[thick][black] (A) -- (C);
    \draw[thick][black] (D) -- (B);
    \fill (A) circle [radius=3.5pt] node[below left] {$A$};
    \fill (B) circle [radius=3.5pt] node[below right] {$B$};
    \fill (C) circle [radius=3.5pt] node[above right] {$C$};
    \fill (D) circle [radius=3.5pt] node[above left] {$D$};
  \endtikzpicture
  \caption{Graph for the complex projective space in $\mathbb{CP}^3/T^2$}
  \label{fig:CP3}
\end{figure}

\section{The product of three-spheres}
\label{product_three_spheres}

It is convenient to view $\mathbb{S}^3 \times \mathbb{S}^3$ as $\mathrm{Sp}(1) \times \mathrm{Sp}(1) \subset \mathbb{H} \times \mathbb{H}$, and recall that $\mathrm{Sp}(1) \cong \mathrm{SU}(2)$. A triple $(h,k,l) \in \mathrm{SU}(2)^3$ acts on $\mathbb{S}^3 \times \mathbb{S}^3$ as $((h,k,l),(p,q)) \mapsto (hpl^{-1},kql^{-1})$. This action is obviously transitive and the stabiliser of the point $(1,1)$ is given by the triples $(h,h,h) \in \mathrm{SU}(2)^3$. We denote this isotropy group by $\mathrm{SU}(2)_{\Delta}$. Therefore $\mathbb{S}^3 \times \mathbb{S}^3$ has the structure of smooth manifold and is diffeomorphic to $\mathrm{SU}(2)^3/\mathrm{SU}(2)_{\Delta}$. 

We follow \cite{diooos} to construct an almost Hermitian structure. We define an almost complex structure $J$ on $\mathbb{S}^3 \times \mathbb{S}^3$ at the point $(p,q)$ as
\begin{equation*}
J_{(p,q)}(X,Y) \coloneqq \tfrac{1}{\sqrt{3}}(X-2pq^{-1}Y,2qp^{-1}X-Y), \qquad (X,Y) \in \mathfrak{sp}(1) \oplus \mathfrak{sp}(1).
\end{equation*}
The standard product metric $\langle{}\cdot{},{}\cdot{}\rangle$ on $\mathbb{S}^3 \times \mathbb{S}^3$ is not invariant under $J$, so we define a metric $g$ as the average of $\langle{}\cdot{},{}\cdot{}\rangle$ and $\langle J{}\cdot{},J{}\cdot{}\rangle$, and normalise it by a factor $1/3$. At the point $(p,q)$ its expression is
\begin{equation*}
g(X,Y) \coloneqq \tfrac{1}{6}(\langle X,Y \rangle + \langle JX,JY \rangle), \quad X,Y \in T_{(p,q)}(\mathbb{S}^3 \times \mathbb{S}^3).
\end{equation*}
Trivially, $J$ is $g$-orthogonal and $(g,J)$ is then an almost Hermitian structure.

Since $\mathbb{S}^3 \times \mathbb{S}^3$ is homogeneous for $\mathrm{SU}(2)^3$ we can work again at the identity $(1,1) \in \mathbb{S}^3 \times \mathbb{S}^3$ to construct a nearly K\"ahler structure. A basis for $\mathfrak{su}(2)^2$ is given by the vectors
\begin{alignat*}{3}
E_1 = (i,0), \quad & E_2 = (j,0), \quad && E_3 = (-k,0), \\
E_4 = (0,i), \quad & E_5 = (0,j), \quad && E_6 = (0,-k).
\end{alignat*}
Hence, a basis of each tangent space $T_{(p,q)}(\mathbb{S}^3 \times \mathbb{S}^3)$ is 
\begin{alignat*}{3}
E_1(p,q) = (pi,0), \quad & E_2(p,q) = (pj,0), \quad && E_3(p,q) = (-pk,0), \\
E_4(p,q) = (0,qi), \quad & E_5(p,q) = (0,qj), \quad && E_6(p,q) = (0,-qk).
\end{alignat*}
\begin{prop}
The forms on $\mathfrak{sp}(1) \times \mathfrak{sp}(1)$ given by
\begin{align*}
\sigma_0 & \coloneqq g_0(J_0{}\cdot{},{}\cdot{}) = \tfrac{2}{3\sqrt{3}}\bigl(e^{14}+e^{25}+e^{36}\bigr), \\
\varphi_0 & \coloneqq \tfrac{4}{9\sqrt{3}}\bigl(e^{126}-e^{135}-e^{156}+e^{234}+e^{246}-e^{345}\bigr), \\
\psi_0 & \coloneqq -J_0\varphi_0 = -\tfrac{4}{27}\bigl(2e^{123}+2e^{456}+e^{135}-e^{156}-e^{234}-e^{126}+e^{246}-e^{345}\bigr),
\end{align*}
satisfy $d\sigma_0 = 3\varphi_0$ and $d\psi_0 = -2\sigma_0 \wedge \sigma_0$. Consequently, the differential forms $\sigma_p \coloneqq g_p(J_p{}\cdot{},{}\cdot{})$, $\psi_{+|p} \coloneqq \varphi_0(p^{-1}{}\cdot{},p^{-1}{}\cdot{},p^{-1}{}\cdot{})$, and $\psi_{-|p} \coloneqq \psi_0(p^{-1}{}\cdot{},p^{-1}{}\cdot{},p^{-1}{}\cdot{})$, define a nearly K\"ahler structure on $\mathbb{S}^3 \times \mathbb{S}^3$.
\end{prop}
\begin{proof}
The differentials of the duals $e^k$ of $E_k$ satisfy $de^i = 2e^{jk}$ for $(ijk)$ cyclic permutation of $(123)$ and $(456)$, whence the result.
\end{proof}

The group $\mathrm{SU}(2)^3$ has rank three, so we have an action of a three-torus $T^3$. A study of this case was performed by Dixon \cite{Dixon}, here we concentrate on actions of two-tori. An element $(t_1,t_2,t_3), t_k= e^{i\vartheta_k}$, of a maximal three-torus $T^3 \subset \mathrm{SU}(2)^3$ acts diagonally on $(g_1,g_2,g_3)\mathrm{SU}(2)_{\Delta}$:
\begin{equation}
\label{three_torus_action}
(t_1,t_2,t_3)(g_1,g_2,g_3)\mathrm{SU}(2)_{\Delta} \coloneqq (t_1g_1,t_2g_2,t_3g_3)\mathrm{SU}(2)_{\Delta}.
\end{equation}
Then it acts on $(p,q) \in \mathbb{S}^3 \times \mathbb{S}^3$ as $(t_1pt_3^{-1},t_2qt_3^{-1})$. Each pair of linearly independent vectors $a=(a_1,a_2)$ in $\mathbb{Z}^3$ yields a discrete group $\Gamma \coloneqq \mathbb{Z}^3 \cap (\mathbb{R}a_1 \oplus \mathbb{R}a_2)$ and a two-torus $T_a^2 = (\mathbb{R}a_1 \oplus \mathbb{R}a_2)/\Gamma$, in our three-torus \nolinebreak $T^3$. The $T^3$-action defined above yields the infinitesimal generators at the point $(p,q) \in \mathbb{S}^3 \times \mathbb{S}^3$, $U_1(p,q) = (ip,0)$, $U_2(p,q) = (0,iq)$, $U_3(p,q) = (-pi,-qi)$, which in terms of the basis $E_1,\dots,E_6$ are
\begin{align*}
U_1(p,q) & = \langle \conjugate{p}ip,i\rangle E_1(p,q)+\langle \conjugate{p}ip,j\rangle E_2(p,q)-\langle \conjugate{p}ip,k\rangle E_3(p,q), \\
U_2(p,q) & = \langle \conjugate{q}iq,i\rangle E_4(p,q)+\langle \conjugate{q}iq,j\rangle E_5(p,q)-\langle \conjugate{q}iq,k\rangle E_6(p,q), \\
U_3(p,q) & = -E_1(p,q)-E_4(p,q).
\end{align*}
A multi-moment map in this case is an equivariant map $\nu \colon \mathbb{S}^3 \times \mathbb{S}^3 \to \Lambda^2 \mathbb{R}^3 \cong \mathbb{R}^3$ (cf.\ \cite{MadsenSwann}). Its three real-valued components correspond to $\nu_i \coloneqq \sigma(U_j,U_k)$, with $(ijk)$ cyclic permutation:
\begin{equation*}
\nu(p,q) = \tfrac{2}{3\sqrt{3}}\bigl(\langle \conjugate{q}iq,i\rangle, \langle \conjugate{p}ip,i\rangle, \langle \conjugate{p}ip,\conjugate{q}iq\rangle\bigr).
\end{equation*}
Pointwise, the generators $U,V$ of the two-torus we are interested in are then rational linear combinations of the $U_i$s:
$$U = a_{11}U_1+a_{12}U_2+a_{13}U_3, \quad V  = a_{21}U_1+a_{22}U_2+a_{23}U_3,$$
so the multi-moment map for the $T^2$-action is $\nu_{\mathbb{S}^3 \times \mathbb{S}^3} \coloneqq \sigma(U,V)$:
$$\nu_{\mathbb{S}^3 \times \mathbb{S}^3}(p,q) = (a_{12}a_{23}-a_{22}a_{13})\nu_1-(a_{11}a_{23}-a_{21}a_{13})\nu_2+(a_{11}a_{22}-a_{12}a_{21})\nu_3.$$

Set $b\coloneqq a_1 \times a_2=(a_{12}a_{23}-a_{13}a_{22},-a_{11}a_{23}+a_{13}a_{21},a_{11}a_{22}-a_{12}a_{21})$ and $x=(x^1,x^2,x^3) \coloneqq (\langle \conjugate{p}ip,i\rangle, \langle \conjugate{p}ip,j\rangle,\langle \conjugate{p}ip,k\rangle),y= (y^1,y^2,y^3) \coloneqq (\langle \conjugate{q}iq,i\rangle, \langle \conjugate{q}iq,j\rangle,\langle \conjugate{q}iq,k\rangle)$. The multi-moment map has then the form 
\begin{equation*}
\nu_{\mathbb{S}^3 \times \mathbb{S}^3}(p,q) = \tfrac{2}{3\sqrt{3}}(b_1y^1+b_2x^1+b_3\langle x,y \rangle).
\end{equation*}
To show how the orbit spaces of the zero level sets look like we consider two examples.

\begin{example}
\label{ex_torus_rsi}
If the two-torus is generated by $U_1,U_2$ its elements have the form $(r,s,1) \in T^2 \subset T^3$. A suitable choice of the basis gives $\nu_{\mathbb{S}^3 \times \mathbb{S}^3}(p,q) = \langle x,y \rangle$. Notice that the $T^2$-action is defined via $(r,s,1)\cdot(p,q)=(rp,sq)$ and it is free (cf.\ Figure \ref{fig:S3xS3} \textbf{A}). The map $(p,q) \in \mathbb S^3 \times \mathbb S^3 \to (x,y) \in \mathbb S^2 \times \mathbb S^2$ is constant on the orbits, thus descends to a map $(\mathbb S^3 \times \mathbb S^3)/T^2 \to \mathbb S^2 \times \mathbb S^2$, which is a diffeomorphism. Therefore, the projection $(x,y) \in \nu_{\mathbb S^3 \times \mathbb S^3}^{-1}(0)/T^2 \to y \in \mathbb{S}^2$ realises the orbit space of the zero level set as a circle bundle over $\mathbb{S}^2$. This can be viewed as the unit tangent bundle over~$\mathbb{S}^2$. 
\end{example}

\begin{example}
\label{ex_torus_rsr}
If the two-torus is generated by $U_1+U_3$ and $U_2$, then its elements have the form $(r,s,r) \in T^2 \subset T^3$. A suitable choice of the basis gives $\nu_{\mathbb S^3 \times \mathbb S^3}(p,q) = \langle x-i,y\rangle$. The explicit action is $(r,s,r)\cdot (p,q)=(rpr^{-1},sqr^{-1})$, and it easily seen to be non-free. We distinguish the cases $y=\pm i$ and $y \neq \pm i$.

In the first case $\langle x-i,y\rangle = 0$ implies $x=i$. Call $\pi\colon \mathbb S^3 \times \mathbb S^3 \to \mathbb S^2 \times \mathbb S^2$ the map sending $(p,q)$ to $(x,y)$. Then $\pi^{-1}(i,i) = \{(e^{i\vartheta},e^{i\varphi}):\vartheta,\varphi \in \mathbb R\}$ and $\pi^{-1}(i,-i) = \{(e^{i\vartheta},e^{i\varphi}j):\vartheta,\varphi \in \mathbb R\}$. For both fibres $T^2$ acts trivially on $e^{i\vartheta}$, and one can then put $e^{i\varphi}=1$, so $q=1$ or $q=j$. Consider the projection $\nu_{\mathbb S^3 \times \mathbb S^3}^{-1}(0)/T^2 \to \mathbb S^2$ given by $(p,q) \mapsto y$ as in the example above. We have just proved that the preimages of $\pm i$ are two circles. A last observation is that the isotropy of each $(e^{i\vartheta},1)$ is the circle whose elements are triples $(r,r,r)$, whereas the isotropy of each $(e^{i\vartheta},j)$ is another circle with elements $(r,r^{-1},r)$ (compare this fact with the graph below, Figure \ref{fig:S3xS3} \textbf{C}).

Now look at the case $y \neq \pm i$. Here we have three subcases: $x= \pm i$ or $x \neq \pm i$. 
The case $x=i$ is trivial, as any $y$ satisfies $\langle x-i,y\rangle=0$, and therefore $\nu_{\mathbb S^3 \times \mathbb S^3}(e^{i\vartheta},q)=0$. Using the torus symmetry we can set $q^1,q^2$ real and non-negative. Therefore, the preimage of $y$ under the map $(p,q) \in \nu_{\mathbb S^3 \times \mathbb S^3}^{-1}(0)/T^2 \to y \in \mathbb S^2$ is a circle. When $x=-i$ then $y=e^{i\vartheta}j$, and again we can set $q^1,q^2$ real and non-negative so that the preimage of each $y$ is a circle.
In the final case one has $\langle x,j\rangle \neq 0$ or $\langle x,k\rangle \neq 0$, consequently $p$ must be of the form $p=p^1+p^2j$ with $p^2 \neq 0$. Similarly, $q=q^1+q^2j$ must have $q^2 \neq 0$. We can then use the torus action 
to make $p^2,q^1,q^2$ real and non-negative, whereas $p^1$ is free as it is fixed by the action. Therefore, the projection $y\colon \nu_{\mathbb S^3 \times \mathbb S^3}^{-1}(0)/T^2 \to \mathbb S^2$ realises again the orbit space of the zero level set as the
unit tangent bundle over the two-sphere $\mathbb S^2$.
\end{example}

Now let us focus on the critical points of $\nu_{\mathbb S^3 \times \mathbb S^3}$, so points $(p,q)$ where $\psi_+(U,V,{}\cdot{})=0$. The generators $U,V$ in terms of $E_1,\dots, E_6$ are
\begin{align*}
U & = (a_{11}\langle \conjugate{p}ip,i\rangle - a_{13})E_1+a_{11}\langle \conjugate{p}ip,j\rangle E_2 - a_{11}\langle \conjugate{p}ip,k\rangle E_3 \\
& \qquad + (a_{12}\langle \conjugate{q}iq,i\rangle-a_{13})E_4+a_{12}\langle \conjugate{q}iq,j\rangle E_5 - a_{12}\langle \conjugate{q}iq,k\rangle E_6, \\
V & = (a_{21}\langle \conjugate{p}ip,i\rangle - a_{23})E_1+a_{21}\langle \conjugate{p}ip,j\rangle E_2 - a_{21}\langle \conjugate{p}ip,k\rangle E_3 \\
& \qquad + (a_{22}\langle \conjugate{q}iq,i\rangle-a_{23})E_4+a_{22}\langle \conjugate{q}iq,j\rangle E_5 - a_{22}\langle \conjugate{q}iq,k\rangle E_6.
\end{align*}

A computation of $\psi_+(U,V,{}\cdot{})_{|(p,q)}$ gives
\begin{align*}
V \chair U \chair e^{126} & = -b_3x^2y^3e^1 + \bigl(b_3x^1y^3+b_1y^3\bigr) e^2 - b_2x^2 e^6, \\
V \chair U \chair e^{315} & = b_3x^3y^2 e^1+\bigl(b_3x^1y^2+b_1y^2\bigr)e^3-b_2x^3 e^5, \\
V \chair U \chair e^{156} & = \big(b_3x^1y^3+b_1y^3 \bigr)e^5 + \bigl(b_3x^1y^2+b_1y^2 \bigr) e^6, \\
V \chair U \chair e^{234} & = -\bigl(b_3x^3y^1+b_2x^3)e^2 -(b_3x^2y^1+b_2x^2)e^3, \\
V \chair U \chair e^{264} & = b_1y^3e^2-b_3x^2y^3e^4-\bigl(b_3x^2y^1+b_2x^2 \bigr)e^6, \\
V \chair U \chair e^{345} & = b_1 y^2e^3+b_3x^3y^2e^4 - \bigl(b_3x^3y^1+b_2x^3 \bigr) e^5.
\end{align*}
Therefore, the equation $\psi_+(U,V,{}\cdot{})_{|(p,q)} = 0$ is equivalent to the following system:
\begin{equation}
\label{critical_system_spheres}
\begin{cases}
b_3(x^3y^2-x^2y^3)=0\\
b_3(x^1y^i-x^iy^1)-b_2x^i=0, & i = 2,3\\
b_3(x^jy^1-x^1y^j)-b_1y^j=0. & j = 2,3.
\end{cases}
\end{equation}
\begin{remark}
As it turns out, $\mathbb{S}^3 \times \mathbb{S}^3$ is the only homogeneous example where saddle points appear and where extrema of the multi-moment map are not symmetric with respect to the origin. We first discuss system \eqref{critical_system_spheres} and then illustrate the various situations assigning explicit values to our parameters $b_1,b_2,b_3$.
\end{remark}
The vectors $x$ and $y$ lie in $\mathbb{S}^2 \subset  \im \mathbb{H}$. We distinguish the cases $b_3=0$ and $b_3\neq 0$. Since not all the $b_i$s vanish, when $b_3=0$ we have three subcases: if $b_1 \neq 0$ and $b_2=0$ then $y=\pm i$, whereas if $b_1=0$ and $b_2\neq 0$ then $x=\pm i$, and finally when $b_1\neq0\neq b_2$ then $x=\pm i,y=\pm i$.
If $b_3 \neq0$ then the first equation yields $x^3y^2-x^2y^3=0$. Summing second and third equations in \eqref{critical_system_spheres} with $i=j$ we get 
\begin{equation}
\label{sums}
b_1y^3+b_2x^3=0=b_1y^2+b_2x^2.
\end{equation}
We end up with three more cases:
\begin{enumerate}
\item If $b_1=b_2=0$ then we obtain at once that $x$ is parallel to $y$, thus $y = \pm x$. 
\item If $b_1 \neq 0$ and $b_2=0$ or $b_1=0$ and $b_2\neq 0$ then $x$ is parallel to $y$ and $y = \pm i$.
\item If $b_1 \neq 0$ and $b_2 \neq 0$ then by \eqref{sums} we get $x^2=-(b_1/b_2)y^2, x^3 = -(b_1/b_2)y^3$, so plugging these solutions in the system one obtains
$$
\begin{cases}
y^2(b_2b_3x^1+b_1b_3y^1+b_1b_2)=0, \\
y^3(b_2b_3x^1+b_1b_3y^1+b_1b_2)=0,
\end{cases}
$$
so either $y=\pm i$ (and then $x=\pm i$) or $b_2b_3x^1+b_1b_3y^1+b_1b_2=0$. 
\end{enumerate}
In the latter case the point $(x^1,y^1) \in \mathbb{R}^2$ lies on the line 
$$r\colon b_2b_3x^1+b_1b_3y^1+b_1b_2=0.$$
On the other hand, since $\lvert x \rvert^2 = 1$ and $x^2=-(b_1/b_2)y^2, x^3 = -(b_1/b_2)y^3$, we have $(x^1)^2 + (b_1^2/b_2^2)\bigl((y^2)^2+(y^3)^2\bigr) = 1$. But $ \lvert y\rvert^2 = 1$ as well, so $(y^2)^2+(y^3)^2=1-(y^1)^2$, and replacing this in the former identity we find a curve 
$$h\colon b_2^2(x^1)^2-b_1^2(y^1)^2 = b_2^2-b_1^2.$$ 
Therefore $(x^1,y^1)$ lies in the intersection between $h$ and $r$. Note that when $b_1 = b_2$ the curve $h$ is the union of the two lines $x^1=\pm y^1$. The slope of $r$ is $-b_2/b_1$ in general, so it is $-1$ when $b_1= b_2$. Since $b_2 \neq 0$ the only non-trivial intersection is between $y^1 = x^1$ and $y^1 = -x^1-b_2/b_3$, which gives $x^1 = y^1 = -b_2/2b_3$. Similarly, when $b_2 = -b_1$ the only non-trivial intersection is between $x^1 = -y^1$ and $y^1 = x^1 - b_2/b_3$, namely $x^1 = -y^1 = b_2/2b_3$. When $b_2 \neq \pm b_1$, the curve $h$ is a hyperbola and its asymptotes have equations $y^1 = \pm (b_2/b_1)x^1$. Since $b_2 \neq 0$ there is a unique intersection between $h$ and $r$ with $x^1 = (b_1^2b_3^2-b_1^2b_2^2-b_2^2b_3^2)/2b_1b_2^2b_3$ and $y^1 = (b_2^2b_3^2-b_1^2b_2^2-b_1^2b_3^2)/2b_1^2b_2b_3$. 

Summing up, in every case we have a uniquely determined solution that may be written as 
\begin{gather*}
x^1 = (b_1^2b_3^2-b_1^2b_2^2-b_2^2b_3^2)/2b_1b_2^2b_3, \quad  y^1 = (b_2^2b_3^2-b_1^2b_2^2-b_1^2b_3^2)/2b_1^2b_2b_3, \\
x^2=-(b_1/b_2)y^2, \quad x^3 = -(b_1/b_2)y^3.
\end{gather*}
Hereafter we list the possible critical values of the multi-moment map:
\begin{enumerate}
\item If at least one $b_i$ vanishes, then 
$$\nu_{\mathbb{S}^3 \times \mathbb{S}^3}(p,q) = \tfrac{2}{3\sqrt{3}}\bigl(y_1 b_1 +x_1 b_2 \pm b_3\bigr),$$
where $x_1b_2= \pm b_2$, $y _1b_1= \pm b_1$.

\item If $b_1,b_2,b_3 \neq 0$, first observe that
\begin{align*}
b_1y^1+b_2x^1 & = \frac{b_2^2b_3^2-b_1^2b_2^2-b_1^2b_3^2}{2b_1b_2b_3}+\frac{b_1^2b_3^2-b_1^2b_2^2-b_2^2b_3^2}{2b_1b_2b_3} = -\frac{b_1b_2}{b_3}.
\end{align*}
It is convenient to write $\langle x,y\rangle$ as $\cos \vartheta$, where $\vartheta$ is the angle between the vectors $x$ and $y$. The system giving critical points and the conditions $x^2=-(b_1/b_2)y^2, x^3 = -(b_1/b_2)y^3$ are saying that $1-\cos^2 \vartheta =  \sin^2\vartheta =  \lVert x \times y \rVert^2 = (b_1^2/b_3^2)\bigl(1-(y^1)^2\bigr)$, namely
\begin{equation*}
\cos^2 \vartheta = 1-\frac{b_1^2}{b_3^2}\left(1-\frac{(b_2^2b_3^2-b_1^2b_2^2-b_1^2b_3^2)^2}{4b_1^4b_2^2b_3^2} \right) = \left(\frac{b_2^2b_3^2-b_1^2b_2^2+b_1^2b_3^2}{2b_1b_2b_3^2}\right)^2.
\end{equation*}
Consequently $\langle x,y \rangle = \cos \vartheta =  \pm (b_2^2b_3^2-b_1^2b_2^2+b_1^2b_3^2)/2b_1b_2b_3^2$, the signs $\pm$ giving two stationary orbits, and
\begin{equation*}
\nu_{\mathbb{S}^3 \times \mathbb{S}^3}(p,q) = \frac{2}{3\sqrt{3}}\left(-\frac{b_1b_2}{b_3}\pm \frac{b_2^2b_3^2-b_1^2b_2^2+b_1^2b_3^2}{2b_1b_2b_3}\right).
\end{equation*}
\end{enumerate}

\begin{example}
\label{example_saddle}
Consider the case $b_3=0$ with $b_1,b_2 \neq 0$. We write $x = \varepsilon_1 i, y = \varepsilon_2 i$, with $\varepsilon_k \in \{ \pm 1\}$. If e.g.\ $a_1 = (2,3,1)$ and $a_2 = (2,3,5)$, then $b=(12,-8,0)$ and $12\varepsilon_2 -8\varepsilon_1 \in \{-20,-4,4,20 \}$, so we have four different non-zero critical values. The points corresponding to the value $4$ are obtained when $x=(1,0,0)=y$ and are actually saddle points. To check this recall that the $T^2$-symmetry allows one to evaluate the multi-moment map on points in $\mathbb{S}^2 \times \mathbb{S}^2$ rather than in $\mathbb{S}^3 \times \mathbb{S}^3$. So considering particular points around $(x,y) = ((1,0,0),(1,0,0))$ we can prove our claim. For example, for $\alpha \neq 0$ and small 
\begin{align*}
\nu_{\mathbb{S}^3 \times \mathbb{S}^3}((\cos\alpha, \sin \alpha, 0),(1,0,0)) & > \nu_{\mathbb{S}^3 \times \mathbb{S}^3}((1,0,0),(1,0,0)), \\
\nu_{\mathbb{S}^3 \times \mathbb{S}^3}((1,0,0),(\cos\alpha, \sin \alpha, 0)) & < \nu_{\mathbb{S}^3 \times \mathbb{S}^3}((1,0,0),(1,0,0)),
\end{align*}
hence $((1,0,0),(0,0,1))$ in $\mathbb{S}^2 \times \mathbb{S}^2$ corresponds to an orbit of saddle points in $\mathbb{S}^3 \times \mathbb{S}^3$. Analogous steps can be repeated for the value $-4$.
\end{example}
\begin{example}
The triple $b=(1,1,2)$ can be obtained using $a_1 = (1,-1,0)$ and $a_2=(1,1,-1)$, so the multi-moment map takes all values between  $-3/2\sqrt{3}$ and $5/6\sqrt{3}$. This shows that maximum and minimum may occur without being symmetric with respect to the origin. 
\end{example}
\begin{example}
The triple $b = (0,0,1)$ occurs for example when $a_1=(1,0,0), a_2=(0,1,0)$, and the extreme values of the multi-moment map are $\pm 2/3\sqrt{3}$. The connected components of the corresponding critical sets in $\mathbb{S}^3 \times \mathbb{S}^3$ are four-dimensional.
\end{example}
Finally, let us construct the graphs. Consider $(t_1,t_2,t_3) \in T^3$, with $t_k = e^{i\vartheta_k}$ for some $\vartheta_k \in \mathbb{R}$ and map each of them into $\mathrm{SU}(2)$ so that $t_i \mapsto \diag(t_i, t_i{}\negthinspace^{-1}) \in \mathrm{SU}(2)$.
The action is then as in \eqref{three_torus_action}. If $(g_1,g_2,g_3)\mathrm{SU}(2)_{\Delta}$ is fixed by $(t_1,t_2,t_3)$ then we have $(t_1g_1,t_2g_2, t_3g_3) = (g_1g,g_2g,g_3g)$  for some $g \in \mathrm{SU}(2)$. Isolating $g$ on one side we find $g_1\negthinspace{}^{-1}t_1g_1 = g_2{}\negthinspace^{-1}t_2g_2$ and $g_1{}\negthinspace^{-1}t_1g_1 = g_3{}\negthinspace^{-1}t_3g_3$, so
\begin{equation*}
t_1 = (g_1g_2{}\negthinspace^{-1})t_2(g_1g_2{}\negthinspace^{-1})^{-1}, \qquad t_1 = (g_1g_3{}\negthinspace^{-1})t_3(g_1g_3{}\negthinspace^{-1})^{-1}.
\end{equation*}
This shows that $t_1$ and $t_2$ are conjugate, as well as $t_1$ and $t_3$. Thus each pair has to have the same eigenvalues. Since $t_1,t_2,t_3$ are diagonal matrices we see that if $t_1 = \diag(e^{i \vartheta},e^{-i\vartheta})$ then $t_k = \diag(e^{i\vartheta},e^{-i\vartheta})$ or $t_k = \diag(e^{-i\vartheta},e^{i\vartheta})$ for $k = 2,3$. This leads us to consider four cases: write $t_1 = t$, then $(t_1,t_2,t_3)$ can be written as either $(t,t,t)$, $(t,t^{-1},t)$, $(t,t,t^{-1})$, or $(t,t^{-1},t^{-1})$. Note there is a discrete stabiliser given by $t=\pm \id$, meaning that the action is not effective. By the usual argument as in the two cases above we can thus ignore it.

In the first case we get $g_2g_1{}\negthinspace^{-1} = \diag(\lambda,\conjugate{\lambda})$ with $\lvert \lambda \rvert = 1$, as $g_2g_1{}\negthinspace^{-1} \in \mathrm{SU}(2)$ commutes with $t$. The same holds for $g_3g_1{}\negthinspace^{-1}$, so we can write $g_3g_1{}\negthinspace^{-1} = \diag(\lambda',\conjugate{\lambda}')$, with $\lvert \lambda' \rvert = 1$. Hence 
\begin{align*}
(g_1,g_2,g_3)\mathrm{SU}(2)_{\Delta} & = (\id, g_2g_1{}\negthinspace^{-1},g_3g_1{}\negthinspace^{-1})\mathrm{SU}(2)_{\Delta} \\
& = (\id, \diag(\lambda, \conjugate{\lambda}), \diag(\lambda',\conjugate{\lambda'}))\mathrm{SU}(2)_{\Delta},
\end{align*}
and the group of elements of the form $(g_1,g_2,g_3)\mathrm{SU}(2)_{\Delta}$ is then isomorphic to $\mathbb{S}^1 \times \mathbb{S}^1 = T^2$. This shows we have a two-torus whose points are fixed by $\mathbb{S}^1$. The other cases are similar. The image of these critical sets in the orbit space $(\mathbb{S}^3 \times \mathbb{S}^3)/T^3$ is given by four points (cf.\ Figure \ref{fig:S3xS3}, {\textbf D}). The saddle points found in Example \ref{example_saddle} lie in these four special $T^3$-orbits.

For every $T^2$ in $T^3$, the stabilizers are still zero- or one-dimensional as $\text{Stab}_{T^2}(p) \subset \text{Stab}_{T^3}(p)$. Thus there are no vertices in our graph, we get only disjoint circles. Further, a $T^2 \subset T^3$ cannot contain all the circles $(t,t,t),(t,t^{-1},t),(t,t,t^{-1}),(t,t^{-1},t^{-1})$. There are three cases: the two-torus may contain none, one or two of the circles above. For example, the first case happens when $T^2$ is of the form $(r,s,1)$, $r,s \in \mathbb{S}^1$, so we get an empty graph and the $T^2$-action is free (cf.\ Figure~\ref{fig:S3xS3} \textbf{A}, and compare with Example \ref{ex_torus_rsi}). If $T^2$ contains triples $(r,rs,rs^2)$, $r,s \in \mathbb{S}^1$, then it contains the circle $(r,r,r)$, so the graph is a single circle. Thirdly, if $T^2$ is of the form $(r,s,r)$, then it contains the circles $(t,t,t)$ and $(t,t^{-1},t)$, but does not include $(t,t,t^{-1})$ and $(t,t^{-1},t^{-1})$, so we get two circles in our graph (cf.\ Figure~\ref{fig:S3xS3} \textbf{B}, \textbf{C}, compare the latter with Example \ref{ex_torus_rsr}).

\begin{figure}
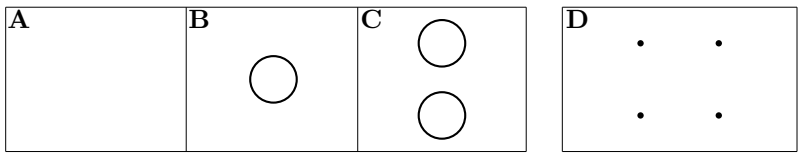

\centering
  \tikzpicture
  [scale=0.8]
  \draw[very thin][black] (-11.25, 1.8) -- (-11.25, -0.6);
  \draw[very thin][black] (-11.25,1.8) -- (-2.6,1.8);
  \draw[very thin][black] (-11.25,-0.6) -- (-2.6,-0.6);
   \draw[very thin][black] (-8.25,-0.6) -- (-8.25,1.8); 
   \draw[very thin][black] (-2,1.8) -- (1.9,1.8);
   \draw[very thin][black] (-2,-0.6) -- (1.9,-0.6);
   \draw (-11.03,1.6) node {\textbf A};
   \draw (-8.03,1.6) node {\textbf B};
   \draw[thick][black] (-6.8,0.6) circle [radius=11pt]; 
   \draw[very thin][black]  (-5.4,-0.6) -- (-5.4,1.8);
   \draw (-5.17,1.6) node {\textbf C};
   \draw[thick][black] (-4,0) circle [radius=11pt];
   \draw[thick][black] (-4,1.2) circle [radius=11pt];
   \draw[very thin][black]  (-2.6,-0.6) -- (-2.6,1.8);
   \draw[very thin][black] (-2,-0.6) -- (-2,1.8);
   \draw (-1.75,1.6) node {\textbf D};
   \draw[very thin][black]  (1.9,-0.6) -- (1.9,1.8);
 \coordinate (A) at (0.6,0);
    \coordinate (B) at (-0.7,0);
    \coordinate (C) at (-0.7,1.2);
    \coordinate (D) at (0.6,1.2);
     \fill (A) circle [radius=1.5pt];
    \fill (B) circle [radius=1.5pt];
    \fill (C) circle [radius=1.5pt];
    \fill (D) circle [radius=1.5pt];
     \endtikzpicture
  \caption{In \textbf{A} the empty graph corresponding to the free $T^2$-action. In \textbf{B} and \textbf{C} the sets of special orbits of the $T^2$-action when this is not free. In \textbf{D} the four points corresponding to the sets of special orbits of the $T^3$-action.}
  \label{fig:S3xS3}
\end{figure}

\providecommand{\bysame}{\leavevmode\hbox to3em{\hrulefill}\thinspace}
\providecommand{\MR}{\relax\ifhmode\unskip\space\fi MR }
\providecommand{\MRhref}[2]{%
  \href{http://www.ams.org/mathscinet-getitem?mr=#1}{#2}
}
\providecommand{\href}[2]{#2}

\small (G. Russo) Department of Mathematics, Geometry and Topology Group, Aarhus University, Ny Munkegade 118, Bldg 1530, DK-8000 Aarhus C, Denmark 

\textit{E\/-mail address}: {\tt giovanni.russo@math.au.dk}\\

\end{document}